\documentclass{article}
\usepackage[margin = 1.2in]{geometry}
\usepackage{amsmath}
\usepackage{amssymb}
\usepackage{amsthm}
\usepackage[english]{babel}
\usepackage{fancyhdr}
\usepackage[colorlinks=true, allcolors=blue]{hyperref}
\usepackage{tikz-cd}
\usepackage{dsfont}
\usepackage{diagbox}

\newtheorem{theorem}{Theorem}[section]
\newtheorem{corollary}[theorem]{Corollary}
\newtheorem{lemma}[theorem]{Lemma}
\newtheorem{prop}[theorem]{Proposition}
\newtheorem{example}[theorem]{Example}

\newtheorem{definition}[theorem]{Definition}
\newtheorem{rmk}[theorem]{Remark}

\DeclareMathOperator{\rk}{rk}
\DeclareMathOperator{\Fund}{Fund}
\DeclareMathOperator{\Span}{Span}
\DeclareMathOperator{\codim}{codim}
\DeclareMathOperator{\ad}{ad}
\DeclareMathOperator{\Ad}{Ad}

\DeclareMathOperator{\Gr}{Gr}
\DeclareMathOperator{\fg}{\mathfrak g}
\DeclareMathOperator{\fgd}{\mathfrak g^{\ast}}
\DeclareMathOperator{\fd}{\mathfrak d}

\DeclareMathOperator{\fh}{\mathfrak h}

\DeclareMathOperator{\fl}{\mathfrak l}
\DeclareMathOperator{\cl}{\mathcal L}
\DeclareMathOperator{\C}{\mathbb C}
\DeclareMathOperator{\Z}{\mathbb Z}

\title{\bf {Wonderful Compactification of a Cartan Subalgebra of a Semisimple Lie Algebra}}
\author{Sam Evens and Yu Li}
\date{\vspace{-5ex}}

\begin{document}

\maketitle


\begin{abstract}
Let $\fh$ be a Cartan subalgebra of a complex semisimple Lie algebra $\fg.$   We define a compactification $\bar \fh$ of $\fh$, which is analogous to the closure $\bar H$ of the corresponding maximal torus $H$ in the adjoint group of $\fg$ in its wonderful compactification, which was introduced and studied by De Concini and Procesi \cite{DCP}.  We observe that $\bar \fh$ is a matroid Schubert variety and prove that the irreducible components of the boundary $\bar \fh - \fh$ of $\fh$ are divisors indexed by root system data.  We prove that $\bar \fh$ is a normal variety and find an affine paving of $\bar \fh,$ where the strata are given by the orbits of $\fh.$   We show that the strata of $\bar \fh$ correspond bijectively to subspaces of the corresponding Coxeter hyperplane arrangement studied by Orlik and Solomon, and prove that the associated posets are isomorphic.   As a consequence, we express the Betti numbers of $\bar \fh$ in terms of well-known combinatorial invariants in the classical cases.   We show that the Weyl group $W$ acts on $\bar \fh$, and describe $H^{\bullet}(\bar \fh, \C)$ as a representation of $W$, and compute the cup product for $H^{\bullet}(\bar \fh, \Z)$.   
\end{abstract}

\tableofcontents

\medskip

\section{Introduction}

The closure $\bar H$ of a maximal torus $H$ of a semisimple complex group $G$ of adjoint type in its {\it wonderful compactification} $X$ was studied by De Concini and Procesi \cite{DCP}.   It is a smooth toric variety with fan given by the Weyl chamber decomposition, and this property is closely related to the smoothness of the wonderful compactification.   In a series of papers \cite{EvLu, EvLu2}, the first author and Lu embedded $X$ inside the {\it variety of Lagrangian subalgebras} of $\fg \oplus \fg$, where $\fg$ is the Lie algebra of $G$ and used the smoothness of $X$ to prove that all irreducible components of the variety of Lagrangian subalgebras are smooth.   Motivations from Poisson geometry suggest that one should also study the variety $\cl$ of Lagrangian subalgebras in a degeneration of $\fg \oplus \fg$ given by the semidirect sum of $\fg$ with its coadjoint representation $\fgd.$  It is elementary to observe that the variety $\cl$ contains a subvariety $\bar {\fg}^{\ast}$ given by the closure of $\fgd.$   While we are currently unable to understand the geometry of $\bar {\fg}^{\ast}$, in this paper we obtain a rather complete understanding of the closure $\bar \fh$ of a Cartan subalgebra $\fh$ of $\fg$ embedded into $\cl.$  We view $\bar \fh$ as the natural additive analogue of $\bar H$, and its study leads in a natural way to the theory of \textit{additive toric varieties} \cite{Cro}.  In support of this perspective, in this paper we prove that $\bar \fh$ is a normal variety with finitely many $\fh$-orbits, which gives an affine paving of $\bar \fh$, and determine a number of invariants of the topology of $\bar \fh.$

In more detail, following ideas from \cite{EvLu, EvLu2}, consider the Lie algebra
$\fd := \fg \ltimes \fgd$, where $\fg$ acts on its dual space $\fgd$ by the coadjoint action, and the factor $\fgd$ is regarded as an abelian Lie algebra.  The Lie algebra $\fd$ comes equipped with a symmetric nondegenerate invariant bilinear form $(~,~)$ defined by $$((x, \phi), (y, \psi)) := \phi(y) + \psi(x),$$ for all $x, y \in \fg$ and $\phi, \psi \in \fgd.$  We call a Lie subalgebra $\fl$ of $\fd$ Lagrangian if $\fl$ is a maximal isotropic subspace of $\fd$ relative to the above symmetric form, and a triple $(\fd, \fl_1, \fl_2)$ is called a {\it Manin triple} if $\fl_1$ and $\fl_2$ are transversal Lagrangian subalgebras of $\fd.$  Recall that Manin triples as above are closely related to {\it Poisson Lie group} structures on a connected group $L_1$ with Lie algebra $\fl_1.$   The set $\cl$ of Lagrangian subalgebras of $\fd$ has the structure of a (reduced) closed subvariety of the Grassmannian $\Gr(\dim \fg, \fd)$  of vector subspaces of $\fd$ with the same dimension as $\fg.$  The variety $\cl$ is referred to as the variety of Lagrangian subalgebras of $\fd.$  The paper \cite{EvLu2} gives a detailed understanding of the variety $\cl(\fg \oplus \fg)$ of Lagrangian subalgebras of $\fg \oplus \fg$, which is a natural setting to study  families of Poisson homogeneous spaces for the \textit{standard} Poisson Lie group structure on $G.$   The variety $\cl$ is in a natural sense a degeneration of $\cl(\fg\oplus\fg)$, but the geometry of $\cl$ seems to be much more difficult to understand.    In a recent paper \cite{EL}, the authors found an unexpected connection between $\cl$ and the abelian ideals of a Borel subalgebra of $\fg$, and it would be of interest to relate the results of \cite{EL} to the present paper.

 We may embed $\fgd$ into $\cl$ via the embedding $$\phi \longmapsto \fl_{\phi}:=\{(x, - \ad_x^*(\phi)): x \in \fg \},$$ and we refer to the closure $\bar {\fg}^{\ast}$ of $\fgd$ in $\cl$ as the wonderful compactification of $\fgd.$  Embed $\fh$ into $\fgd$ via the Killing form.  Then the composition $$\fh \lhook\joinrel\longrightarrow \fgd \lhook\joinrel\longrightarrow \cl$$ embeds $\fh$ into $\cl.$  The closure $\bar \fh$ of the image of this embedding in $\cl$ is called the wonderful compactification of $\fh$ and is the main object of study of this paper.  As a consequence of our construction, we will see (Proposition \ref{ambient}) that $\bar \fh$ is a special case of the matroid Schubert varieties studied in \cite{AB} and \cite{BHMPWI}.

Write $\Phi$ for the root system for $(\fg, \fh).$  Let $\mathcal S$ be the set of \textit{closed} root subsystems of $\Phi$ of rank $\rk \Phi - 1.$  We view $\mathcal S$ as a poset, where the partial order is inclusion.  A maximal element of the poset $(\mathcal S, \subseteq)$ is called a \textit{good} root subsystem of $\Phi.$  For a good root subsystem $\Phi'$ of $\Phi$, we let $\fg'$ be a semisimple Lie subalgebra of $\fg$ and $\fh'$ a Cartan subalgebra of $\fg'$ such that the root system for $(\fg', \fh')$ is isomorphic to $\Phi'.$  We prove the following result about the structure of $\bar \fh.$

\begin{theorem}[see Theorem \ref{m1'}] \label{m1} 
$\ $
\begin{enumerate}
    \itemsep 0em
    \item There is a bijection
\begin{align*}
C: \{\text{good root subsystems of} ~ \Phi\} \longrightarrow \{\text{irreducible components of} ~ \bar \fh - \fh\}, ~ \Phi' \longmapsto C(\Phi').
\end{align*}
In particular, the boundary $\bar \fh - \fh$ of $\fh$ has as many irreducible components as there are good root subsystems of $\Phi.$

\item For a good root subsystem $\Phi'$ of $\Phi$, the irreducible component $C(\Phi')$ is isomorphic as a variety to $\bar \fh'$, the wonderful compactification of $\fh'.$  In particular, we have $$\dim C(\Phi') = \dim \fh - 1,$$ so that the boundary $\bar \fh - \fh$ of $\fh$ is of pure dimension $\dim \fh -1.$
\end{enumerate}
\end{theorem}


To analyze the singularities of $\bar \fh$, we prove the following result.

\begin{theorem}[see Theorem \ref{m3'}] \label{m3} 
The set $$\fh \cup (\bigcup \limits_{\substack {\Phi' ~ \text{is a good} \\ \text{root subsystem}}} \fh')$$ consists of regular points of $\bar \fh$, where $\fh'$ is defined as above.  In particular, the variety $\bar \fh$ is regular in codimension one.
\end{theorem}

Using results in \cite{AB}, one deduces that the variety $\bar \fh$ is Cohen-Macaulay.  Hence, it satisfies Serre's condition (S2).  In view of Theorem \ref{m3} and Serre's criterion for normality, we have

\begin{theorem}[see Theorem \ref{normal}] \label{m4} 
The variety $\bar \fh$ is normal.
\end{theorem}


We further consider $k$-step good root subsystems of $\Phi$, which are defined inductively as good root subsystems of $(k-1)$-step good root subsystems of $\Phi$, where $\Phi$ is regarded as a $0$-step good root subsystem of $\Phi.$
We then apply the above results to find an affine paving 
$$\bar \fh = \bigsqcup \mathring C(\Psi),$$ where the union is over all $k$-step good root subsystems $\Psi$ of $\Phi$ for $k=0, \dots, r := \dim \fh$ in Corollary \ref{strat}.  View $\fh$ as an additive algebraic group.  We prove that the action of $\fh$ on itself by addition extends to an action of $\fh$ on $\bar \fh$.  Further, the subvarieties $\mathring C(\Psi)$ are the $\fh$-orbits of $\bar \fh.$  In particular, there are finitely many $\fh$-orbits in $\bar \fh$, and $\mathring C(\Phi)$ is the open $\fh$-orbit which is $\fh$-equivariantly isomorphic to $\fh$.
Due to the above results, we may regard $\bar \fh$ as an additive analogue of a toric variety.

As a consequence of the above affine paving, much of the structure of $H^{\bullet}(\bar \fh, \Z)$ is determined by the combinatorics of $k$-step good root subsystems.  To study this combinatorics, we relate the order relation given by closures of strata in $\bar \fh$ to the \textit{Coxeter arrangements} studied by Orlik and Solomon.  More precisely, we give an isomorphism of posets between the poset consisting of strata of $\bar \fh$ with order relation given by closure, and the poset of subspaces of a Coxeter arrangement.  As a consequence, the Betti numbers of $\bar \fh$ can be deduced from the number of subspaces of given rank in the corresponding Coxeter arrangement, and these are computed in \cite{OT} and \cite{OS1}.  In more detail, for a root system $\Phi$, we let $f(\Phi,k) := \rk H^{2(r-k)}(\bar \fh, \Z)$, which is the number of $k$-step good root subsystems.  We let $S(n,k)$ denote the {\it Stirling number of the second kind}, which counts the number of partitions of the set $\{ 1, \dots, n \}$ into $k$ nonempty parts.

\begin{theorem}[see Theorem \ref{thm:Betti}] \label{m6}
 For $r \in \mathbb N$ and $0 \le k \le r$,  the Betti number  $$f(A_r, k) = S(r+1, k+1).$$
\end{theorem}

In Theorem \ref{thm:Betti}, we give analogous formulas for types $B$, $C$, and $D$ (note that the bijection with Coxeter arrangements implies that the answer is the same for types $B$ and $C$).  We also give formulas for the Betti numbers of $\bar \fh$ in the exceptional cases,  which can be deduced from tables in \cite{OS1}, but we also computed them using the software SageMath.

We further show that the above isomorphism of posets is compatible with natural Weyl group actions.   We use these results to determine the Weyl group representation on $H^{\bullet}(\bar \fh, \C)$ in Corollary 
\ref{rep}.  Finally, we determine the cup product structure of $H^{\bullet}(\bar \fh, \Z)$ in terms of a natural basis given by the stratification, which we denote by $\xi_X$, where $X$ is a subspace in the intersection lattice $L(\mathcal A)$ of the Coxeter arrangement corresponding to a stratum.  Our result is the same as in \cite{HW}, but our method is much more elementary.

\begin{theorem}[see Theorem \ref{cupprod}] \label{m7} 
    For $X, Y \in L(\mathcal A)$, the cup product structure of $H^{\bullet}(\bar \fh, \Z)$ is given by
    \begin{align*}
        \xi_X \smile \xi_Y =
        \begin{cases}
            \xi_{X \cap Y}  & ~ \text{if} ~ X ~ \text{is transversal to} ~ Y \\
            0 & ~ \text{else}.
        \end{cases}
    \end{align*}
In particular, $H^{\bullet}(\bar \fh, \Z)$ is generated by its degree $2$ component.
\end{theorem}

We discuss the genesis of this paper, and its relation to other papers in the literature.   We established most of the results of this paper without being aware of the notion of matroid Schubert variety from \cite{AB} and \cite{BHMPWI}, but have incorporated results from this literature where it is useful to streamline the exposition.   In particular, we use results from \cite{AB} in an essential way to prove normality of $\bar \fh.$     Since our interest is focused on Lie theoretic aspects of $\bar \fh,$ the relation of $\bar \fh$ to $\bar H,$ and potential applications to the geometry of the resulting compactification of $\fgd,$ we generally have not emphasized perspectives from matroid Schubert varieties.    However, the stratification we describe for $\bar \fh$, which we define using root subsystems, may also be described as a standard stratification for matroid Schubert varieties, and, as stated above, Theorem \ref{m7} is a special case of the main result of \cite{HW}, although our proof is much simpler.  We also note that there is an open affine subvariety of $\bar \fh$ which meets all singular strata, and which is an example of the \textit{reciprocal varieties} introduced by Proudfoot and Speyer \cite{PS}.   In an earlier draft, we were unaware of the computations in \cite{OS1} and \cite{OT}, and we are grateful to the referee and Matt Douglass for informing us of these computations and of the connection between $k$-step good root subsystems and parabolic root systems.  As a result, we previously developed a variant of the work of Borel and de Siebenthal from \cite{BDS} in order to count good root subsystems, and this earlier draft  included a lengthy study of the combinatorics of $k$-step good root subsystems in order to compute the Betti numbers for $\bar \fh.$   In view of \cite{OT}, we have omitted this combinatorics, along with our variant of the work of Borel and de Siebenthal.

Due to the analogy with $\bar H,$ we view our results showing that $\bar \fh$ is normal and has finitely many $\fh$-orbits as indicating that $\bar \fh$ should be regarded as an additive toric variety.  To develop this theory, it would be desirable if each additive toric variety is uniquely determined by a set of combinatorial data such that the combinatorial data determining $\bar \fh$ is the Coxeter arrangement.  See Colin Crowley's paper \cite{Cro} for general results about additive toric varieties.   Forthcoming related work by the second author, Ana Balibanu, and Crowley constructs a flat degeneration such that the generic fiber is the toric variety associated to a rational central essential hyperplane arrangement and the special fiber is a non-reduced union of matroid Schubert varieties.   In this context, $\bar \fh$ is a component of the special fiber for the family with generic fiber $\bar H$, and this family is expected to be a subfamily of the degeneration of $\cl (\fg\oplus\fg)$ to $\cl.$  In further related work, we note that the paper \cite{IKLPR} gives a combinatorial construction of the real locus of $\bar \fh$ as a permutohedron modulo parallel faces for the Cartan subalgebra of a general complex semisimple $\fg.$   In forthcoming work of the second author with Leo Jiang, the authors give an analogous construction for any matroid Schubert variety of a central essential real hyperplane arrangement.

This paper is organized as follows.  In Section 2, we set up basic notation and define objects of interest in the study of $\bar \fh.$  In Section 3 we prove that $\bar \fh$ is normal and has an affine paving given by $\fh$-orbits.  In particular, Theorems \ref{m1}, \ref{m3} and \ref{m4} are proved in this section.  In Section 4, we relate our stratification of $\bar \fh$ to Coxeter arrangements, and use this to compute the Betti numbers of $\bar \fh$ for an arbitrary semisimple Lie algebra.  Theorem \ref{m6} is proved in this section.  In Section 5, we compute the Weyl group representation on $H^{\bullet}(\bar \fh, \C)$ and  determine the cohomology ring of $\bar \fh$ in terms of the basis given by the strata.  Theorem \ref{m7} is proved in this section. 

{\bf Acknowledgments.}  We are grateful to Ana Balibanu, Colin Crowley, Victor Ginzburg, Boming Jia, Leo Jiang, Joel Kamnitzer, Jacob Matherne, Eckhard Meinrenken, Nick Proudfoot, Nick Salter, and Shuddhodan Vasudevan for many stimulating discussions.  We would also like to thank Jacob Matherne and Nick Proudfoot for encouraging us to write up this paper, and we are especially grateful to Matt Douglass for his careful reading of an earlier draft of this paper and for his incisive and helpful comments, and the referee for several comments which have significantly improved our paper.

\section{Construction of \texorpdfstring{$\bar \fh$}{hbar}} \label{won}
In this section, we define $\bar \fh$ and give an embedding of $\bar \fh$ into a product of projective lines.  As a consequence, we show that $\bar \fh$ is a matroid Schubert variety.

Throughout this paper, we work over the complex numbers.  Let $G$ be a connected semisimple algebraic group of adjoint type and $\fg := \text{Lie} (G)$ the Lie algebra of $G.$  Write $\fgd$ for the dual space of $\fg.$  The adjoint (resp. coadjoint) action of $G$ on $\fg$ (resp. $\fgd$) will be denoted by $\Ad$ (resp. $\Ad^{\ast}$).  The adjoint (resp. coadjoint) action of $\fg$ on $\fg$ (resp. $\fgd$) will be denoted by $\ad$ (resp. $\ad^{\ast}$).  Define the algebraic group $D$ to be the semidirect product $G \ltimes \fgd$, where $G$ acts on $\fgd$ by the coadjoint action.  Specifically, as a set $D$ is the Cartesian product $G \times \fgd$, and group multiplication in $D$ is given by
\[
(g, \phi) (h, \psi) := (gh, \Ad^{\ast}_{h^{-1}} \phi + \psi) \quad \forall g, h \in G \text{ and } \phi, \psi \in \fgd.
\]
Inversion in $D$ is given by
\[
(g, \phi)^{-1} := (g^{-1}, - \Ad^{\ast}_g \phi) \quad \forall g \in G \text{ and } \phi \in \fgd.
\]

Consider the Lie algebra $\fd$ of $D$, which is the semidirect sum $\fg \ltimes \fgd.$  As a vector space, $\fd$ is the Cartesian product $\fg \times \fgd.$  The Lie bracket in $\fd$ is given by
\[
[(x, \phi), (y, \psi)] := ([x, y], \ad^{\ast}_x \psi - \ad^{\ast}_y \phi) \quad \forall x, y \in \fg \text{ and } \phi, \psi \in \fgd.
\]

The exponential map $\text{Exp}: \fd \rightarrow D$ is given by
\[
\text{Exp} (x, \phi) = \bigl( \text{exp}(x), \sum_{i=0}^{\infty} \frac{(-1)^{i}}{(i+1)!} (\ad^*_x)^i \phi \bigr) \quad \forall x \in \fg \text{ and } \phi \in \fgd,
\]
where $\text{exp}: \fg \rightarrow G$ is the exponential map for $G.$  This formula corrects an error in the formula for the exponential map that we presented in \cite[Equation (2)]{EL}, but this mistake does not affect the rest of that paper.  By abuse of notation, the adjoint action of $D$ on $\fd$ is also denoted by $\Ad.$  Using the formulas above, one verifies that the adjoint action of $D$ on $\fd$ is given by 
\begin{align} \label{adj}
\Ad_{(g, \phi)} (x, \psi) = (\Ad_g x, - \Ad^{\ast}_g \ad^{\ast}_x \phi + \Ad^{\ast}_g \psi) \quad \forall g \in G, x \in \fg \text{ and } \phi, \psi \in \fgd.
\end{align}

The algebraic group $D$ has the following important algebraic subgroups.  It is easy to check that the homomorphism of algebraic groups $$G \longrightarrow D, ~ g \longmapsto (g, 0)$$ is an embedding of $G$ into $D.$  We will regard $G$ as an algebraic subgroup of $D$ via this homomorphism.  View $\fgd$ as an additive algebraic group.  It is easy to check that the homomorphism of algebraic groups $$\fgd \longrightarrow D, ~ \phi \longmapsto (e, \phi),$$ where $e$ stands for the identity element of $G$, is an embedding of $\fgd$ into $D.$  We will regard $\fgd$ as an algebraic subgroup of $D$ via this homomorphism.  At the Lie algebra level, we view $\fg$ (resp. $\fgd$) as a Lie subalgebra of $\fd$ via the embedding $\fg \rightarrow \fd, ~ x \mapsto (x, 0)$ (resp. $\fgd \rightarrow \fd, ~ \phi \mapsto (0, \phi)$).

Introduce a bilinear form $(~, ~): \fd \otimes \fd \rightarrow \mathbb C$ by defining
\[
((x, \phi), (y, \psi)) := \phi (y) + \psi (x) \quad \forall x, y \in \fg \text{ and } \phi, \psi \in \fgd.
\]
It is easy to verify that the bilinear form $(~, ~)$ is symmetric, nondegenerate and $D$-invariant.  Here, $D$-invariance means that
\begin{align} \label{inv}
    (\Ad_{(g, \phi)} (x, \psi), \Ad_{(g, \phi)} (y, \chi)) = ((x, \psi), (y, \chi)) \quad \forall g \in G, x,y \in \fg \text{ and } \phi, \psi, \chi \in \fgd.
\end{align}

Let $V$ be a finite dimensional vector space and $(~, ~)$ a nondegenerate bilinear form on $V.$  A vector subspace $W$ of $V$ is called \textit{isotropic} if $(w, w') = 0$ for all $w, w' \in W.$  An isotropic subspace $W$ of $V$ is \textit{Lagrangian} if $W$ is maximal, with respect to inclusion, among all isotropic subspaces of $V.$  When $V$ is even dimensional,  this condition is equivalent to saying that $\dim W = \frac 12 \dim V.$

A \textit{Lagrangian subalgebra} of $\fd$ is a Lie subalgebra $\fl$ of $\fd$ which is also a Lagrangian vector subspace of $\fd$ with respect to the nondegenerate bilinear form $(~, ~)$ on $\fd$ introduced above.  For example, it is easy to check that $\fg$ and $\fgd$ are Lagrangian subalgebras of $\fd.$  Moreover, if $\fl$ is a Lagrangian subalgebra of $\fd$ and $(g, \phi)$ is an element of $D$, then $\Ad_{(g, \phi)} \fl$ is also a Lagrangian subalgebra of $\fd$ by Equation (\ref{inv}).

Let $\Gr(n, \fd)$ be the Grassmannian of $n$-dimensional vector subspaces of $\fd$, where $n := \dim \fg = \frac 12 \dim \fd.$  For an $n$-dimensional vector subspace of $\fd$, the property of being a Lie subalgebra (resp. being an isotropic vector subspace) of $\fd$ is a closed condition.  Therefore, the set of Lagrangian subalgebras of $\fd$ has a natural structure of a reduced closed subvariety of $\Gr(n, \fd)$, to be denoted by $\cl.$  Since $\Gr(n, \fd)$ is a projective variety, the variety $\mathcal L$ is projective as well.  We will call $\mathcal L$ the {\it variety of Lagrangian subalgebras} of $\fd.$  By Equation (\ref{inv}), the image of the map $$D \times \cl \longrightarrow \Gr(n, \fd), ~ ((g, \phi), \fl) \longmapsto \Ad_{(g, \phi)} \fl$$ is in $\cl$ and, hence, gives an action of the algebraic group $D$ on the variety $\cl.$  By abuse of notation, this action will be denoted by $\Ad$.  The following lemma is easy to prove.

\begin{lemma} \label{stabfg}
The stabilizer in $D$ at the point $\fg$ of $\cl$ is given by $$\text{Stab}_D (\fg) = \{(g, 0): g \in G\} \cong G.$$
\end{lemma}

From Lemma \ref{stabfg}, we see that we can identify $D/G$ with the $D$-orbit in $\cl$ through the point $\fg.$  Since the subgroups $\fgd$ and $G$ of $D$ intersect only at the identity element $(e, 0)$ of $D$, the morphism $$\fgd \longrightarrow D/G$$ sending an element $\phi$ of $\fgd$ to the coset of $(e,\phi)$ is an isomorphism of varieties.  Therefore, we get an embedding $$\fgd \lhook\joinrel\longrightarrow{} \cl, ~ \phi \longmapsto \Ad_{(e,\phi)} \fg$$ of $\fgd$ into $\cl.$  More concretely, by Equation (\ref{adj}), this embedding sends an element $\phi$ of $\fgd$ to the Lagrangian subalgebra $$\{(x, - \ad^{\ast}_x \phi): x \in \fg\}$$ of $\fd.$

\begin{definition}
The closure in $\cl$ (equivalently, in $\Gr(n, \fd)$) of the image of the embedding above is called the \textit{wonderful compactification} of $\fgd,$ which will be denoted by $\bar{\fg}^{\ast}$.
\end{definition}

In the following example, we compute $\bar \fg^{\ast}$ in the case where $\fg$ is $\mathfrak {sl}_2.$

\begin{example} \label{ex}
Let $\fg = \mathfrak {sl}_2.$  Let
\begin{align*}
    E := 
    \begin{pmatrix}
        0 & 1 \\
        0 & 0
    \end{pmatrix}, ~ 
    F := 
    \begin{pmatrix}
        0 & 0 \\
        1 & 0
    \end{pmatrix}, ~ 
    H := 
    \begin{pmatrix}
        1 & 0 \\
        0 & -1
    \end{pmatrix}
\end{align*}
be the usual basis of $\fg.$  Identify $\fgd$ with $\fg$ via the trace form.  So, the dual basis of $\fgd (\cong \fg)$ is $\{F, E, \frac H2\}.$  Then we get a basis $$\{v_1 := (E, 0), ~ v_2 := (F, 0), ~ v_3 := (H, 0), ~ v_4 := (0, F), ~ v_5 := (0, E), ~ v_6 := (0, \frac H2)\}$$ of $\fd.$

For an element $\phi = aF + bE + c \frac H2$ of $\fgd$, one computes:
\begin{align*}
    - \ad_E^{\ast} \phi = cE - 2a \frac H2, ~ - \ad_F^{\ast} \phi = -cF + 2b \frac H2, ~ - \ad_H^{\ast} \phi = 2aF - 2bE.
\end{align*}
Hence, the embedding above sends $\phi$ to the Lagrangian subalgebra $${\rm Span} (v_1 + cv_5 - 2av_6, v_2 - cv_4 + 2bv_6, v_3 + 2av_4 - 2bv_5).$$

Consider the Pl\"ucker embedding $\Gr(3, \fd) \hookrightarrow \mathbb P(\wedge^3 \fd).$  The vector space $\wedge^3 \fd$ has an ordered basis $$\{v_1 \wedge v_2 \wedge v_3, ~ v_1 \wedge v_2 \wedge v_4, ~ v_1 \wedge v_2 \wedge v_5, ~ v_1 \wedge v_2 \wedge v_6, ~ v_1 \wedge v_3 \wedge v_4, \ldots, v_4 \wedge v_5 \wedge v_6\}.$$  With respect to this basis, the composition $$f: \fgd \lhook\joinrel\longrightarrow{} \cl \subseteq \Gr(3, \fd) \lhook\joinrel\longrightarrow \mathbb P(\wedge^3 \fd) \cong \mathbb P^{19}$$ sends $\phi = aF + bE + c\frac H2$ to $$[1:2a:-2b:0:c:0:-2b:2bc:-4ab:4b^2:0:c:-2a:2ac:-4a^2:4ab:c^2:-2ac:2bc:0].$$  From this we get a commutative diagram
\begin{equation*}
    \begin{tikzcd}
        \fgd \arrow[rd, "f"] \arrow[d, "\jmath"] & \\
        \mathbb P^3 \arrow[r, "\tilde f"] & \mathbb P^{19},
    \end{tikzcd}
\end{equation*}
where $$\jmath (aF + bE + c\frac H2) := [1:a:b:c]$$ and
\begin{align*}
    \tilde f([x_0:x_1:x_2:x_3]) := [x_0^2: 2x_0x_1: -2x_0x_2: 0: x_0x_3: 0: -2x_0x_2: 2x_2x_3: -4x_1x_2: 4x_2^2: 0: \\ x_0x_3: -2x_0x_1: 2x_1x_3: -4x_1^2: 4x_1x_2: x_3^2: -2x_1x_3: 2x_2x_3: 0].
\end{align*}

We claim that $\tilde f$ is a closed embedding.  In fact, one can check that $\tilde f$ is equal to the composition
\begin{equation*}
    \begin{tikzcd}
        \mathbb P^3 \arrow[r, "\Delta"] & \mathbb P^3 \times \mathbb P^3 \arrow[r, "{\rm Sg}"] & \mathbb P^{15} \arrow[r, "L"] & \mathbb P^{19},
    \end{tikzcd}
\end{equation*}
where $\Delta$ is the diagonal embedding, ${\rm Sg}$ is the Segre embedding $${\rm Sg}([x_0:x_1:x_2:x_3], [y_0:y_1:y_2:y_3]) := [x_0y_0: x_0y_1: x_0y_2: x_0y_3: x_1y_0: \cdots : x_3y_3]$$ and $L$ is the linear map
\begin{align*}
    L([z_0: \cdots : z_{15}]) := [z_0: 2z_1: -2z_2: 0: z_3: 0: -2z_8: 2z_{11}: -4z_6: 4z_{10}: 0: z_{12}: -2z_4: 2z_7: -4z_5: \\ 4z_9: z_{15}: -2z_{13}: 2z_{14}: 0].
\end{align*}
Since $\Delta, {\rm Sg}$ and $L$ are closed embeddings, so is $\tilde f.$

Now, since the image of $\jmath$ is dense in $\mathbb P^3$, we see that the closure of the image of $f$ can be identified with $\mathbb P^3.$  Therefore, we have obtained an isomorphism of varieties $\overline {\mathfrak {sl}}^{\ast}_2 \cong \mathbb P^3.$
\end{example}

Example \ref{ex} is somewhat misleading in the sense that, in general, $\bar \fg^{\ast}$ is expected to be a singular variety.  It was pointed out in \cite{EL} that the geometry of $\cl$ is far more intricate than that of the variety of Lagrangian subalgebras of $\fg \oplus \fg$, equipped with the bilinear form
\[
(\fg \oplus \fg) \otimes (\fg \oplus \fg) \longrightarrow \mathbb C, ~ ((x_1,y_1), (x_2,y_2)) \longmapsto \kappa(x_1,x_2) - \kappa(y_1,y_2),
\]
where $\kappa$ stands for the Killing form.  In particular, the geometry of $\bar \fg^{\ast}$ is expected to be more complicated than that of the wonderful compactification $X$ of $G$ studied by De Concini and Procesi \cite{DCP}.  It was observed in the literature, cf. \cite{EvJ}, that to study the geometry of $X$, it is usually helpful if one first studies the geometry of the wonderful compactification $\bar H$ of a Cartan subgroup $H$ of $G.$  As a first step towards understanding the geometry of $\bar \fg^{\ast}$, we introduce an analogue of $\bar H$ as follows.

Fix a Cartan subalgebra $\fh$ of $\fg.$  We have an embedding $$\fh \lhook\joinrel\longrightarrow{} \fgd, ~ h \longmapsto \kappa(h, -).$$  Composing this with the embedding $\fgd \hookrightarrow \cl$, we get an embedding $${\rm emb}: \fh \lhook\joinrel\longrightarrow{} \fgd \lhook\joinrel\longrightarrow{} \cl \subseteq \Gr(n, \fd).$$

\begin{definition}
The closure in $\cl$ (equivalently, in $\Gr(n, \fd)$) of the image of the above embedding of $\fh$ is called the \textit{wonderful compactification} of $\fh,$ which will be denoted by $\bar \fh$.
\end{definition}

\begin{example}
Let $\fg = \mathfrak {sl}_2.$  Retain the notation of Example \ref{ex}.  In order to determine $\bar \fh$, we only need to set the complex numbers $a, b$ in Example \ref{ex} to zero.  Then we see that $\bar \fh$ is isomorphic to the closure of the set $$\{[1: 0: 0: 0: c: 0: 0: 0: 0: 0: 0: c: 0: 0: 0: 0: c^2: 0: 0: 0]: c \in \mathbb C\}$$ in $\mathbb P^{19}$, which, in turn, is isomorphic to the closure of the set $$\{[1:c:c^2]: c \in \mathbb C\}$$ in $\mathbb P^2.$  The last variety is clearly the parabola $$\{[x_0:x_1:x_2]: x_0x_2 - x_1^2 = 0\}$$ in $\mathbb P^2.$  Hence $\bar \fh$ is isomorphic to $\mathbb P^1$ when $\fg$ is $\mathfrak {sl}_2.$
\end{example}

Identify $\fg$ with $\fgd$ via $$\fg \cong \fgd, ~ x \longmapsto \kappa(x, -).$$  This induces an identification of $\Gr(n, \fd)$ with $\Gr(n, \fg \times \fg).$  Under this identification, if an element $\phi$ of $\fgd$ is identified with $y \in \fg$, then the Lagrangian subalgebra $$\{(x, - \ad^{\ast}_x \phi): x \in \fg\}$$ of $\fd$ is identified with the vector subspace $$\{(x, [y,x]): x \in \fg\}$$ of $\fg \times \fg.$  The $\fgd$-action on $\fd$ from Equation (\ref{adj}) induces a $\fg$-action on $\fg \times \fg$ given by $(z,(x,u))\mapsto (x,[z,x]+u).$
Let $\Phi$ be the root system for $(\fg, \fh).$  For each root $\lambda \in \Phi$, choose a root vector $e_{\lambda}.$  When $y$ is in $\fh$, the vector space $\{(x, [y,x]): x \in \fg\}$ is clearly spanned by $\fh \oplus 0$ and $(e_{\lambda}, \lambda(y) e_{\lambda})$ for all $\lambda \in \Phi.$  This proves the following

\begin{lemma} \label{embed}
The composition $$\fh \lhook\joinrel\longrightarrow \cl \subseteq \Gr(n, \fd) \cong \Gr(n, \fg \times \fg)$$ sends an element $h$ of $\fh$ to ${\rm Span}(\fh \oplus 0, (e_{\lambda}, \lambda(h) e_{\lambda}): \lambda \in \Phi).$
\end{lemma}

By abuse of notation, the composition in Lemma \ref{embed} will also be denoted by ${\rm emb}.$  Note that ${\rm emb}$ is $\fh$-equivariant for the translation action of $\fh$ on itself and the induced action of $\fh$ on $\Gr(n, \fg \times \fg).$

Let $V$ be a vector space and $\{v_1, \ldots, v_m\}$ a basis of $V.$  Define a morphism
\[
f: (\mathbb P^1)^m \longrightarrow \Gr(m, V \times V), ~ f \bigl( ([x_{i,0}: x_{i,1}] )_{1 \le i \le m} \bigr) := {\rm Span} ((x_{i,0}v_i, x_{i,1}v_i): 1 \le i \le m).
\]
It is easy to see that this morphism is well-defined.  Arguing in local coordinates, it is also easy to see that this morphism is a closed embedding.

Consider the case where $V = \fg.$  For this, consider a basis $$\{h_1, \ldots, h_r, e_{\lambda}: \lambda \in \Phi\}$$ of $\fg,$ where $\{h_1, \ldots, h_r\}$ is a basis of $\fh.$  The morphism ${\rm emb}: \fh \hookrightarrow \Gr(n, \fg \times \fg)$ factors through the morphism $f: (\mathbb P^1)^n \hookrightarrow \Gr(n, \fg \times \fg)$ explained above.  In fact, there is a commutative diagram
\begin{equation*}
    \begin{tikzcd}
        \fh \arrow[d, "{\rm emb}'"] \arrow[dr, "{\rm emb}"] \\
        (\mathbb P^1)^n \arrow[r, "f"] & \Gr(n, \fg \times \fg),
    \end{tikzcd}
\end{equation*}
where ${\rm emb}'$ sends an element $h$ of $\fh$ to the point of $(\mathbb P^1)^n$ whose $h_i$-component is $[1:0]$ for all $1 \le i \le r$ and whose $\lambda$-component is $[1:\lambda (h)]$ for all $\lambda \in \Phi.$   Further, if $\fh$ acts on $(\mathbb P^1)^n$ by the formula 
$$y\cdot ([x_{i,0}:x_{i,1}])_{i=1, \ldots, r} \times ([x_{\lambda,0},x_{\lambda,1}])_{\lambda \in \Phi} = ([x_{i,0}:x_{i,1}])_{i=1, \ldots, r} \times ([x_{\lambda,0},\lambda(y)x_{\lambda,0}+x_{\lambda,1}])_{\lambda \in \Phi},$$
then the morphisms ${\rm emb}'$ and $f$ are $\fh$-equivariant.

We identify the open set consisting of $x=[x_0:x_1]$ of  ${\mathbb P^1}$ with $x_0\not= 0$ with $\C$ using the map $x\mapsto  \frac{x_1}{x_0} \in \C$, and write $\infty$ for the point $[0:1].$

Choose a decomposition $\Phi = \Phi^+ \sqcup \Phi^-$ of $\Phi$ into the sets of positive and negative roots.  Let $d$ be the cardinality of $\Phi^+.$  Introduce a closed embedding $l: (\mathbb P^1)^d \rightarrow (\mathbb P^1)^n$ by sending the point $(x_{\lambda})_{\lambda \in \Phi^+}$ of $(\mathbb P^1)^d$ to the point of $(\mathbb P^1)^n$ whose $h_i$-component is $0$ for all $1 \le i \le r$, whose $\lambda$-component is $x_{\lambda}$ and whose $(-\lambda)$-component is $-x_{\lambda}$ for all $\lambda \in \Phi^+.$  Then it is clear that ${\rm emb}': \fh \hookrightarrow (\mathbb P^1)^n$ factors through $l: (\mathbb P^1)^d \rightarrow (\mathbb P^1)^n$, namely, there is a commutative diagram
\begin{equation*}
    \begin{tikzcd}
        \fh \arrow[d, "{\rm emb}''"] \arrow[dr, "{\rm emb}'"] \\
        (\mathbb P^1)^d \arrow[r, "l"] & (\mathbb P^1)^n,
    \end{tikzcd}
\end{equation*}
where ${\rm emb}''$ sends an element $h$ of $\fh$ to the point of $(\mathbb P^1)^d$ whose $\lambda$-component is $[1:\lambda (h)]$ for all $\lambda \in \Phi^+.$  These arguments prove the following

\begin{prop} \label{ambient}
As a variety, $\bar \fh$ is isomorphic to the closure in $(\mathbb P^1)^d$ of the image of the morphism ${\rm emb}'': \fh \hookrightarrow (\mathbb P^1)^d.$
\end{prop}

Another way of stating Proposition \ref{ambient} is as follows.  View $\mathbb C^d$ as the vector space of $(\Phi^+)$-tuples of complex numbers.  View $\mathbb C^d$ also as the affine chart of $(\mathbb P^1)^d$ consisting of points all of whose components are not equal to $\infty.$  The linear map $$\fh \longrightarrow \mathbb C^d, ~ h \longmapsto (\lambda(h))_{\lambda \in \Phi^+}$$ is clearly injective.  Via this injective linear map, we view $\fh$ as a vector subspace of $\mathbb C^d.$  Then Proposition \ref{ambient} says that $\bar \fh$ is isomorphic to the closure of the vector subspace $\fh$ of $\mathbb C^d$ in $(\mathbb P^1)^d.$  

\begin{definition} \label{matroidschubert}
The closure of a vector subspace $V$ of $\C^d$ in $(\mathbb P^1)^d$ is called a matroid Schubert variety.
\end{definition}

Thus, $\bar \fh$ is a matroid Schubert variety.  Matroid Schubert varieties were first considered in \cite{AB} and play a crucial role in the proof of the \textit{top-heavy conjecture} by Braden, Huh, Matherne, Proudfoot, and Wang \cite{BHMPWII}. 

\begin{rmk} \label{haction}
We let $\fh$ act on $(\mathbb P^1)^d$ by the formula 
$$
y\cdot([x_{\lambda,0},x_{\lambda,1}])_{\lambda \in \Phi^+} = ([x_{\lambda,0},\lambda(y)x_{\lambda,0} + x_{\lambda,1}])_{\lambda \in \Phi^+}
$$
so that the morphisms ${\rm emb}''$ and $l$ are $\fh$-equivariant.
\end{rmk}

\section{Algebro-geometric Properties of \texorpdfstring{$\bar \fh$}{hbar}}

In this section, we prove that $\bar \fh$ is a normal variety by showing it is regular in codimension one and using results from \cite{AB} to show that $\bar \fh$ is Cohen-Macaulay. In the process, we show that there are finitely many $\fh$-orbits on $\bar \fh$, and that the orbits give an affine paving of $\bar \fh.$  For this reason, it is reasonable to think of $\bar \fh$ as an additive analogue of a toric variety.   We also extend the action of the Weyl group $W$ on $\fh$ to a $W$-action on $\bar \fh$, and we explicitly compute $\bar \fh$ in the case of $\mathfrak{sl}_3$, in which case $\bar \fh$ is a singular hypersurface in $(\mathbb P^1)^3.$


\subsection{Defining Ideal and Boundary Components of \texorpdfstring{$\bar \fh$}{hbar}}

Let $S := \mathbb C[x_{\lambda,0}, x_{\lambda,1}: \lambda \in \Phi^+]$ be the multi-homogeneous coordinate ring of $(\mathbb P^1)^d.$  The ring $S$ is $(\mathbb Z_{\ge 0})^d$-graded so that the variables $x_{\lambda,0}, x_{\lambda,1}$, for $\lambda \in \Phi^+$, have multi-degree $$(0, \ldots, 0, 1, 0, \ldots, 0),$$ where the only $1$ appears in the $\lambda$-component.  An element of $S$ is said to be multi-homogeneous if all its terms have the same multi-degree which we refer to as ``degree'' in the sequel.  It is easy to see that every element $p$ of $S$ can be uniquely written in the form $$p = \sum \limits_{u \in (\mathbb Z_{\ge 0})^d} p_u,$$ where $p_u$ is multi-homogeneous of degree $u$ for each $u \in (\mathbb Z_{\ge 0})^d.$  The $p_u$'s will be referred to as the multi-homogeneous components of $p.$  Let $J$ be an ideal of $S.$  We say that $J$ is a multi-homogeneous ideal if it can be generated by multi-homogeneous polynomials.  Equivalently, $J$ is multi-homogeneous if whenever an element of $S$ is in $J$, so are all its multi-homogeneous components.

As in projective geometry, for each multi-homogeneous ideal $J$ of $S$, the vanishing locus $V(J)$ of $J$ in $(\mathbb P^1)^d$ is well-defined.  
Choose multi-homogeneous generators $p_1, \ldots, p_k$ of $J.$  Then the set $$\{x = ([x_{\lambda,0}: x_{\lambda,1}])_{\lambda \in \Phi^+} \in (\mathbb P^1)^d: p_1 \bigl( (x_{\lambda,0},x_{\lambda,1})_{\lambda \in \Phi^+} \bigr) = 0, \ldots, p_k \bigl( (x_{\lambda,0},x_{\lambda,1})_{\lambda \in \Phi^+} \bigr) = 0\}$$ is well-defined and independent of the choice of multi-homogeneous generators of $J.$  This set is the vanishing locus $V(J).$  Our first goal is to find a multi-homogeneous ideal $J$ of $S$, as simple as possible, whose vanishing locus is $\bar \fh$, and use it to analyze the boundary $\bar \fh - \fh$ of $\fh.$

Let $R := \mathbb C[x_{\lambda}: \lambda \in \Phi^+]$ be the coordinate ring of $\mathbb C^d,$ which we regard as an affine open set of $(\mathbb P^1)^d.$  For $P \in R$, the multi-homogenization of $P$ is the multi-homogeneous polynomial $P^h \in S$ such that $$P \bigl( (\frac{x_{\lambda,1}}{x_{\lambda,0}})_{\lambda \in \Phi^+} \bigr) = \frac{P^h \bigl( (x_{\lambda,0}, x_{\lambda,1})_{\lambda \in \Phi^+} \bigr)}{\prod \limits_{\lambda \in \Phi^+} x_{\lambda,0}^{a_{\lambda}}},$$ where $a_{\lambda} \in \mathbb Z_{\ge 0}$ for all $\lambda \in \Phi^+$ and the numerator and denominator of the right-hand side have no common prime divisors.  For an ideal $I$ of $R$, its multi-homogenization is the ideal $$I^h := \langle P^h: P \in I \rangle$$ of $S$ generated by the multi-homogenization of elements of $I.$  We remark that, in general, if $P_1, \ldots, P_k$ are generators of $I$, the multi-homogeneous ideal $\langle P_1^h, \ldots, P_k^h \rangle$ could be \textit{properly} contained in $I^h.$

Let $I(\fh) \lhd R$ be the ideal consisting of all polynomials vanishing on $\fh$, where we have regarded $\fh$ as a closed subvariety of $\mathbb C^d$ via ${\rm emb}''.$  By Proposition \ref{ambient}, it is clear that $\bar \fh$ is isomorphic to the vanishing locus $V(I(\fh)^h)$ of the multi-homogenization of $I(\fh).$  We aim to find a multi-homogeneous ideal $J$ (possibly properly) contained in $I(\fh)^h$ such that the vanishing locus of $J$ is $\bar \fh$, at least set theoretically, and computation with $J$ is easier than with $I(\fh)^h.$

\begin{definition}\label{multihom}
Define $J(\Phi)$ to be the multi-homogeneous ideal $$\langle P^h: P ~ \text{is a linear polynomial in} ~ R ~ \text{that vanishes on} ~ \fh\rangle$$ of $S.$
\end{definition}

Write $Z \subseteq (\mathbb P^1)^d$ for the vanishing locus of $J(\Phi).$  It is clear from the definition of $J(\Phi)$ that $$Z \supseteq \bar \fh \supseteq \fh.$$  To analyze the difference $Z - \fh$, we need to introduce some notation.  A root subsystem $\Phi'$ of $\Phi$ is called \textit{closed} if for all $\lambda, \mu \in \Phi'$ with $\lambda + \mu \in \Phi$, one has $\lambda + \mu \in \Phi'.$  Let $(x_{\lambda})_{\lambda \in \Phi^+}$ be a point of $Z - \fh \subseteq (\mathbb P^1)^d.$  Define 
\begin{align}\label{findef}
{\rm Fin} := \{\lambda \in \Phi^+: x_{\lambda} \neq \infty\}.
\end{align}
The following lemma can be easily deduced from the fact that if $\lambda, \mu, \lambda+\mu \in \Phi^+$, then $x_{\lambda+\mu}-x_{\lambda}-x_{\mu}$ vanishes on $\fh.$

\begin{lemma} \label{fin}
The set ${\rm Fin} \cup (-{\rm Fin})$ is a closed root subsystem of $\Phi.$
\end{lemma}

Denote by $\mathcal S$ the set of all closed root subsystems of $\Phi$ of rank $r - 1,$ where $r = \rk \Phi.$  The set $\mathcal S$, together with inclusion, is a poset.

\begin{definition}
A \textit{good root subsystem} of $\Phi$ is a maximal member of the poset $(\mathcal S, \subseteq).$
\end{definition}

\begin{example} \label{maxgood}
Let $\Pi = \{ \alpha_1, \dots, \alpha_r \}$ be a system of simple roots.  Then for $i=1, \dots, r,$ the root subsystem generated by $\Pi - \{ \alpha_i \}$
is trivially a good root subsystem.
\end{example}

Let $$P = \sum\limits_{\lambda \in \Phi^+} c_{\lambda} x_{\lambda}$$ be a linear polynomial in $R.$  The \textit{support} $\text{supp}(P)$ of $P$ is defined to be the set $$\{\lambda \in \Phi^+: c_{\lambda} \neq 0\}.$$

\begin{lemma}\label{goodsupport}
Let $\Phi'$ be a good root subsystem of $\Phi$ and $P \in R$ a linear polynomial that vanishes on $\fh.$  Then the following are the only possibilities:
\begin{enumerate}
    \itemsep 0em
    \item Either ${\rm supp}(P) \subseteq \Phi'$; or
    \item There are at least two elements of ${\rm supp}(P)$ that are not contained in $\Phi'.$
\end{enumerate}
\end{lemma}

\begin{proof}
Suppose that $$P = \sum\limits_{\lambda \in \Phi^+} c_{\lambda} x_{\lambda}$$ vanishes on $\fh$ and $\lambda_0$ is the only positive root such that $c_{\lambda_0} \neq 0$ and $\lambda_0 \notin \Phi'.$  Since $P$ vanishes on $\fh$, we must have $$\sum\limits_{\lambda \in \Phi^+} c_{\lambda} \lambda = 0.$$  Hence $\lambda_0$ is a linear combination of elements of $\Phi'.$  It follows that the set of elements of $\Phi$ which can be expressed as an integral linear combination of $\Phi'$ and $\lambda_0$ is a closed root subsystem of $\Phi$ of rank $r - 1$, which properly contains $\Phi'.$  This contradicts the assumption that $\Phi'$ is maximal in $\mathcal S.$
\end{proof}

For each good root subsystem $\Phi'$ of $\Phi$, choose, once and for all, a semisimple Lie subalgebra $\fg'$ of $\fg$ and a Cartan subalgebra $\fh'$ of $\fg'$ such that the root system for $(\fg', \fh')$ is isomorphic to $\Phi'.$  We may choose $(\fg', \fh')$ such that $\fh'$ is a Lie subalgebra of $\fh.$  Fix a good root subsystem $\Phi'$ of $\Phi.$  A linear polynomial $P \in R$ that vanishes on $\fh$ is said to be \textit{of type I} if $\text{supp}(P) \subseteq \Phi'$, and is said to be \textit{of type II} if there are at least two elements of $\text{supp}(P)$ which are not contained in $\Phi'.$  Consider a point $x = (x_{\lambda})_{\lambda \in \Phi^+}$ of $(\mathbb P^1)^d$ such that $x_{\lambda} = \infty$ for all $\lambda \notin \Phi'.$  Then, for every $P$ of type II, the multi-homogenization $P^h$ automatically vanishes at $x.$  Hence, the point $x$ is in $Z$ if and only if $P^h$ vanishes at $x$ for all $P$ of type I.  Put $(\Phi')^+ := \Phi' \cap \Phi^+.$  Then $x$ is in $Z$ if and only if $(x_{\lambda})_{\lambda \in (\Phi')^+} \in (\mathbb P^1)^{(\Phi')^+}$ is in the vanishing locus of the multi-homogeneous ideal $J(\Phi')$, defined similarly as in Definition \ref{multihom}.

\begin{definition}\label{stratumphi}
Let $\Phi'$ be a good root subsystem of $\Phi.$  Define $C({\Phi'})$ to be the subset $$\{(x_{\lambda})_{\lambda \in \Phi^+} \in (\mathbb P^1)^d: x_{\lambda} = \infty ~ \text{for all} ~ \lambda \notin \Phi' ~ \text{and} ~ (x_{\lambda})_{\lambda \in (\Phi')^+} \in V(J(\Phi'))\}$$ of $Z.$
\end{definition}

By definition, Lemma \ref{goodsupport}, and our analysis above, it is clear that $$C({\Phi'}) \subseteq Z - \fh$$ for all good root subsystems $\Phi'.$

\begin{theorem} \label{mainthm}
As algebraic sets, we have $$Z = \bar \fh.$$
\end{theorem}

\begin{proof}
We induct on the rank $r$ of $\fg.$  The statement is easy when $r = 1.$  We assume that $r > 1$ and the statement has been proved for $1, \ldots, r - 1.$

For each good root subsystem $\Phi'$ of $\Phi$, the induction hypothesis implies that $C(\Phi')$ can be identified with $\bar \fh'.$  Since $\fh'$ is $(r - 1)$-dimensional and irreducible, so are $\bar \fh'$ and $C(\Phi').$  Hence $C(\Phi')$ is an $(r - 1)$-dimensional irreducible closed subset of $Z$ which is disjoint from $\fh.$  It follows that $C(\Phi')$ is an irreducible component of $Z - \fh.$  Therefore, we get a map
\[C: \{\text{good root subsystems of} ~ \Phi\} \longrightarrow \{\text{irreducible components of} ~ Z - \fh\}: \Phi' \longmapsto C(\Phi').
\]

We first show that $C$ is injective.  Let $\Phi', \Phi''$ be distinct good root subsystems.  Then there exists a positive root $\lambda_0 \in \Phi' - \Phi''.$  Let $x = (x_{\lambda})_{\lambda \in \Phi^+} \in C(\Phi') \cap C({\Phi''}).$  Since $\lambda_0 \notin \Phi''$, we have $x_{\lambda_0} = \infty.$  This, together with the fact $\lambda_0 \in \Phi'$, implies that $x$ is in the boundary $C(\Phi') - \fh' \cong \bar \fh' - \fh'$ of $\fh'.$  Since $$\dim (\bar \fh' - \fh') < \dim (\fh'),$$ we see that $$\dim (C(\Phi') \cap C(\Phi'')) < \dim (\fh') = \dim (C(\Phi')) = \dim (C(\Phi'')).$$  It follows that $C(\Phi') \cap C(\Phi'')$ is a proper subset of both $C(\Phi')$ and $C(\Phi'').$  Hence $C(\Phi')$ and $C(\Phi'')$ are distinct.

Next we show that $C$ is surjective.  Let $x = (x_{\lambda})_{\lambda \in \Phi^+} \in Z - \fh \subseteq (\mathbb P^1)^d.$  Define ${\rm Fin}$ as in Equation (\ref{findef}).  Then ${\rm Fin} \cup (-{\rm Fin})$ is a closed root subsystem of $\Phi$ by Lemma \ref{fin}.  Suppose that $\rk ({\rm Fin} \cup (-{\rm Fin})) = r.$  Then there exist linearly independent positive roots $\lambda_1, \ldots, \lambda_r$ such that $x_{\lambda_i} \neq \infty$ for all $1 \le i \le r.$  For each $\lambda \in \Phi^+$, write $\lambda = \sum \limits_{i = 1}^r b_i \lambda_i$ for some $b_1, \ldots, b_r \in \mathbb Q.$  Then $$x_{\lambda} - \sum \limits_{i = 1}^r b_i x_{\lambda_i}$$ is a linear polynomial in $R$ that vanishes on $\fh.$  Hence the $\lambda$-component $x_{\lambda}$ of $x$ is also not equal to $\infty.$  Therefore, none of the components of $x$ is equal to $\infty$, i.e., $x \in \fh$, a contradiction.  Consequently, we see that $\rk ({\rm Fin} \cup (-{\rm Fin})) \le r - 1.$  Therefore, ${\rm Fin} \cup (-{\rm Fin})$ is contained in a good root subsystem $\Phi'.$  It then follows from the definition of ${\rm Fin}$ and $C(\Phi')$ that $x \in C(\Phi').$  We have proved that every point of $Z - \fh$ is contained in some $C(\Phi')$, so the map $C$ is surjective.

Finally we show that $Z = \bar \fh.$  We only need to prove that every point of $Z - \fh$ is in the closure of $\fh$ with respect to the classical topology.  Since $C$ is a bijection, we only need to prove that, for each good root subsystem $\Phi'$, every point of $C(\Phi')$ is in the closure of $\fh$ with respect to the classical topology.  By the induction hypothesis, we need only prove that, for each good root subsystem $\Phi'$, every point of the set
\begin{equation}\label{stratum}
\begin{split}
\mathring C(\Phi') := \{x = (x_{\lambda})_{\lambda \in \Phi^+} \in (\mathbb P^1)^d: ~ & x_{\lambda} = \infty ~ \forall \lambda \notin \Phi', x_{\lambda} \neq \infty ~ \forall \lambda \in \Phi' \\ & \text{and} ~ (x_{\lambda})_{\lambda \in (\Phi')^+} ~ \text{is in the vanishing locus of} ~ J(\Phi')\}
\end{split}
\end{equation}
is in the closure of $\fh$ with respect to the classical topology.  Fix such a $\Phi'$ and $x \in \mathring C(\Phi').$  Since $\rk \Phi' = r - 1$, there exists $\lambda_0 \in \Phi^+ - \Phi'$ such that $${\rm Span}_{\mathbb Q} (\lambda_0, \Phi') = {\rm Span}_{\mathbb Q} (\Phi).$$  Choose linearly independent positive roots $\lambda_1, \ldots, \lambda_{r-1}$ in $\Phi'.$  Then every $\lambda \in \Phi$ can be uniquely expressed as
\begin{align} \label{expression}
\lambda = q_0 \lambda_0 + \sum\limits_{i=1}^{r-1} q_i \lambda_i,    
\end{align}
where $q_0, q_1, \ldots, q_{r-1}$ are rational numbers.  Since $\Phi'$ is a good root subsystem of $\Phi$, one has $$\lambda \in \Phi' ~ \text{if and only if} ~ q_0 = 0.$$  It is clear that, for each $t \in \mathbb C$, there exists a unique point $x_t$ of $\fh$ whose $\lambda_0$-component is $t$ and whose $\lambda_i$-component is $x_{\lambda_i}$ for all $1 \le i \le r - 1.$  Now it follows from Equation (\ref{expression}) that $$\lim \limits_{t \rightarrow \infty} x_t = x.$$
\end{proof}

In the course of proving Theorem \ref{mainthm}, we have proved

\begin{theorem} \label{m1'}
$\ $
\begin{enumerate}
    \itemsep 0em
    \item The map
    \[
    C: \{\text{good root subsystems of} ~ \Phi\} \longrightarrow \{\text{irreducible components of} ~ \bar \fh - \fh\}: \Phi' \longmapsto C(\Phi')
    \]
    is a bijection.  In particular, the boundary $\bar \fh - \fh$ of $\fh$ has as many irreducible components as there are good root subsystems of $\Phi.$
    \item For a good root subsystem $\Phi'$ of $\Phi$, $C(\Phi')$ is isomorphic as a variety to $\bar \fh'$, the wonderful compactification of $\fh'.$  In particular, we have $$\dim C(\Phi') = \dim \fh - 1,$$ so that the boundary $\bar \fh - \fh$ of $\fh$ is of pure dimension $\dim \fh -1.$
\end{enumerate}
\end{theorem}

From now on, we view each $\fh'$ as a subset of $\bar \fh$ via the composition 
\begin{align*} 
\fh' \lhook\joinrel\longrightarrow \bar \fh' \cong C(\Phi') \lhook\joinrel\longrightarrow \bar \fh.
\end{align*}
We remark that the ideal $J(\Phi)$ that we used to compute the boundary of $\fh$, a priori, could be properly contained in $I(\fh)^h.$  However, it is proved in \cite[Theorem 1.3(a)]{AB} that these two ideals are actually equal.

\begin{definition}
    Let $\Phi$ be a root system of rank $r.$  For any $k \in [1, r]$, we inductively define a $k$-step good root subsystem of $\Phi$ to be a good root subsystem of a $(k - 1)$-step good root subsystem of $\Phi$, where a $0$-step good root subsystem is understood to be $\Phi$ itself.
\end{definition}

For each $k$-step good root subsystem $\Psi$, we define $\mathring C(\Psi)$ as in Equation (\ref{stratum}) and define $C(\Psi)$ as in Definition \ref{stratumphi}.

\begin{corollary}\label{strat}
We have $$\bar \fh = \bigsqcup \mathring C(\Psi),$$ where the union is over all $k$-step good root subsystems $\Psi$ of $\Phi$ for $k=0, \dots, r$, and the $\mathring C(\Psi)$ are the $\fh$-orbits on $\bar \fh.$  This is an affine paving of $\bar \fh.$
\end{corollary}

\begin{proof}
We use induction on $r$ and note that the case where $r=1$ is clear.  Note first that $\fh = \mathring C(\Phi)$ is the open stratum.  By Theorem \ref{m1'}, the irreducible components of $\bar \fh - \fh$ are the $C({\Phi^\prime})$, where $\Phi^\prime$ is a 1-step good root subsystem of $\Phi$ and each $C({\Phi'})$ is isomorphic to $\bar \fh'.$
By induction, each $C({\Phi^{\prime}})$ has a stratification by $\mathring C(\Psi)$, where $\Psi$ ranges over $k$-step good root subsystems of $\Phi'$, which are $(k+1)$-step good root subsystems of $\Phi.$  This proves the first assertion, and the second assertion follows by a similar induction.
\end{proof} 

We denote the resulting stratification of $\bar \fh$ by ${\mathcal C}.$  We are grateful to Matt Douglass for pointing out the relation between good and parabolic root subsystems.

\begin{definition}
    Let $\Psi$ be a closed root subsystem of $\Phi$.  We say that $\Psi$ is parabolic if its Dynkin diagram can be obtained from that of $\Phi$ by removing a set of nodes (and all edges adjacent to them).
\end{definition}

\begin{prop} \label{Prop:goodPara}
    The good root subsystems of $\Phi$ are precisely the maximal elements of the poset $(\{\text{proper parabolic root subsystems of } \Phi\}, \subseteq).$
\end{prop}

\begin{proof}
    Let $\Psi$ be a good root subsystem.  Define $\Psi' := (\Span_{\mathbb Q} \Psi) \cap \Phi.$  It is evident that $\Psi'$ is a closed root subsystem of $\Phi$ of rank $r-1$ which contains $\Psi.$  By maximality of $\Psi,$ we must have $\Psi = \Psi'.$  By \cite[Proposition 24, $\S$VI.1.7]{Bou}, $\Psi'$ is a parabolic root subsystem of $\Phi.$  Hence, $\Psi$ is a maximal proper parabolic root subsystem.

    On the other hand, let $\Psi$ be a maximal proper parabolic root subsystem of $\Phi.$  Then, by definition, $\Psi$ is a closed root subsystem.  It is evident that the rank of $\Psi$ is $r-1.$  Since $\Psi$ contains every root in its $\mathbb Q$-span, it is maximal among all closed root subsystems of rank $r-1.$  It follows that $\Psi$ is good.
\end{proof}

\begin{corollary} \label{simcl}
Let $\{\alpha_1, \ldots, \alpha_r\}$ be a system of simple roots.  Take any $1 \le i \le r$ and remove from the Dynkin diagram of $\Phi$ the node corresponding to $\alpha_i.$  Then the root subsystem of $\Phi$ generated by the simple roots corresponding to the remaining nodes is a good root subsystem.  Conversely, up to the action of $W$, all good root subsystems can be obtained in this way.
\end{corollary}

The proof is immediate from Proposition \ref{Prop:goodPara}, since maximal proper parabolic root subsystems are precisely those given by omitting one simple root, up to $W$-conjugacy.
 
In a previous version of this paper, we were unaware of the connection with parabolic root subsystems, and gave a construction of good root subsystems using a variant of the Borel-de Siebenthal algorithm \cite{BDS}.   Proposition \ref{Prop:goodPara} makes this construction unnecessary, so we have omitted it.  We thank Matt Douglass and the referee for their help with this issue.

\subsection{Normality of \texorpdfstring{$\bar \fh$}{hbar}}

We now proceed to show $\bar \fh$ is regular in codimension one.  It is evident that every point of $\fh$ is regular.  Let $x = (x_{\lambda})_{\lambda \in \Phi^+} \in (\mathbb P^1)^d.$  We consider affine charts of  $(\mathbb P^1)^d$ containing $x$ of the form  $\prod \limits_{\lambda \in \Phi^+} U_{\lambda}$, where for each $\lambda \in \Phi^+$, $U_{\lambda}$ is either $\mathbb C$ or $\mathbb P^1 - \{0\}$ and $x_{\lambda} \in U_{\lambda}.$  Given an affine chart $U$ containing $x = (x_{\lambda})_{\lambda \in \Phi^+}$, for $\lambda \in \Phi^+$, we let $z_{\lambda}$ be the inhomogeneous coordinate on $U_{\lambda}.$  Namely, we define
\begin{align*}
    z_{\lambda} :=
    \begin{cases}
        x_{\lambda,1}/x_{\lambda,0} & ~ \text{if} ~ U_{\lambda} = \mathbb C \\
        x_{\lambda,0}/x_{\lambda,1} & ~ \text{if} ~ U_{\lambda} = \mathbb P^1 - \{0\}.
    \end{cases}
\end{align*}
Let $R_U := \mathbb C[z_{\lambda}: \lambda \in \Phi^+]$ be the coordinate ring of $U$ and $\mathfrak p \lhd R_U$ the ideal of polynomials in $R_U$ that vanish on $\bar \fh \cap U.$  Since $\fh$ is irreducible, so is $\bar \fh.$  It follows that $\bar \fh \cap U$ is irreducible and, hence, equidimensional.  Therefore, we can apply the Jacobian criterion to prove regularity of $x.$  Concretely, to prove that $x$ is regular, we need to find $d - r$ polynomials in $\mathfrak p$ such that their differentials evaluated at $x$ are linearly independent.

\begin{theorem} \label{m3'}
$\ $
\begin{enumerate}
\itemsep 0em
\item Every point of the set $$\fh \cup (\bigcup \limits_{\substack {\Phi' ~ \text{is a good} \\ \text{root subsystem}}} \fh')$$ is a regular point of $\bar \fh.$ 
\item  $\bar \fh$ is regular in codimension one.
\end{enumerate}
\end{theorem}

\begin{proof}
Fix a good root subsystem $\Phi'$ and a point $x = (x_{\lambda})_{\lambda \in \Phi^+}$ of $\fh'.$  Retain the notation above.

Choose linearly independent positive roots $\lambda_1, \ldots, \lambda_{r-1}$ of $\Phi'.$  Then for each positive root $\lambda \in \Phi'$, we have a unique expression $$\lambda = \sum\limits_{i=1}^{r-1} a_i \lambda_i,$$ where each $a_i$ is a rational number.  Then $$z_{\lambda} - \sum\limits_{i=1}^{r-1} a_i z_{\lambda_i}$$ is an element of $\mathfrak p.$  Such a polynomial is said to be a type (i) polynomial.  The number of polynomials of type (i) is $|(\Phi')^+| - (r - 1)$, where $|(\Phi')^+|$ stands for the cardinality of $(\Phi')^+.$

Since $\rk \Phi' = r - 1 < \rk \Phi$, there exists $\lambda_0 \in \Phi^+ - \Phi'.$  Suppose that $\lambda_0 \in \text{Span}_{\mathbb Q} (\Phi').$  Then the root subsystem of $\Phi$ consisting of those elements of $\Phi$ which can be expressed as an integral linear combination of the elements of $\Phi'$ and $\lambda_0$ is clearly closed, of rank $r - 1$ and properly contains $\Phi'.$  This contradicts the assumption that $\Phi'$ is good.  Hence the root $\lambda_0 \notin \text{Span}_{\mathbb Q} (\Phi').$  It follows that $\text{Span}_{\mathbb Q} (\Phi) = \text{Span}_{\mathbb Q} (\Phi', \lambda_0).$  Now, for each $\lambda \in \Phi^+ - \Phi'$ with $\lambda \neq \lambda_0$, we have a unique expression $$\lambda = d_0 \lambda_0 + \sum\limits_{i=1}^{r-1} d_i \lambda_i,$$ where each $d_i$ is a rational number and $d_0 \neq 0.$  Therefore, the polynomial $$z_{\lambda_0} - d_0 z_{\lambda} - z_{\lambda_0} z_{\lambda} \sum\limits_{i=1}^{r-1} d_i z_{\lambda_i}$$ is an element of $\mathfrak p.$  Such a polynomial is said to be a type (ii) polynomial.  The number of polynomials of type (ii) is $d - |(\Phi')^+| - 1.$

All together, we have defined $$\bigl( |(\Phi')^+| - (r - 1) \bigr) + (d - |(\Phi')^+| - 1) = d - r$$ polynomials of $\mathfrak p$ (of types (i) and (ii)).  To show that $x$ is a regular point of $\bar \fh$, it suffices to show that the differentials of these polynomials, evaluated at $x$, are linearly independent.  We regard the differential of a polynomial as a column vector in the following manner.  For each $\lambda \in \Phi^+$, the entry of this column vector in the row corresponding to $\lambda$ is the partial derivative of the polynomial with respect to $z_{\lambda}.$  For a polynomial of type (i), its differential has the feature that all the nonzero entries are in the rows corresponding to those $\lambda$'s with $\lambda \in (\Phi')^+.$  Moreover, each such differential contains a unique nonzero entry in the rows corresponding to those $\lambda$'s with $\lambda \in (\Phi')^+ - \{\lambda_1, \ldots, \lambda_{r-1}\}$, and these nonzero entries, as the polynomial varies within the set of polynomials of type (i), are in different rows.  For the polynomial $$z_{\lambda_0} - d_0 z_{\lambda} - z_{\lambda_0} z_{\lambda} \sum\limits_{i=1}^{r-1} d_i z_{\lambda_i}$$ of type (ii), since when evaluating at $x$, we set the variables $z_{\lambda_0}, z_{\lambda}$ to zero, the only nonzero entries of the evaluation of its differential at $x$ are in the rows corresponding to $\lambda_0$ and $\lambda.$  Moreover, these nonzero entries are $1$ in the row corresponding to $\lambda_0$ and $-d_0 (\neq 0)$ in the row corresponding to $\lambda.$  Recall that $\lambda_0$ and $\lambda$ are elements of $\Phi^+ - \Phi'.$  Now it is clear that the differentials of the chosen polynomials (of types (i) and (ii)), evaluated at $x$, are linearly independent.  This proves the first assertion.

By Theorem \ref{m1'}, for each good root subsystem $\Phi'$, the complement of $\fh'$ in $C(\Phi')$ has dimension $r - 1 - 1 = r - 2.$  Again by Theorem \ref{m1'}, the complement of $$\fh \cup (\bigcup \limits_{\substack {\Phi' ~ \text{is a good} \\ \text{root subsystem}}} \fh')$$ in $\bar \fh$ is contained in $$\bigcup \limits_{\substack {\Phi' ~ \text{is a good} \\ \text{root subsystem}}} (C(\Phi') - \fh').$$  Hence, the singular locus of $\bar \fh$ has codimension at least two.  This proves the second assertion.
\end{proof}

To prove that $\bar \fh$ is normal, we use the following results.  Recall that $S$ is the multi-homogeneous coordinate ring of $(\mathbb P^1)^d.$

\begin{theorem} \cite[Theorem 1.1]{AB} \label{CM}
The ring $S/I(\fh)^h$ is Cohen-Macaulay.
\end{theorem}

\begin{theorem} \cite[Theorem 2.1]{BK} \label{tensor}
Let $A, B$ be $\mathbb C$-algebras such that $A \otimes_{\mathbb C} B$ is Noetherian.  Then $A \otimes_{\mathbb C} B$ is Cohen-Macaulay if and only if so are $A$ and $B.$
\end{theorem}

\begin{theorem} \label{normal}
The variety $\bar \fh$ is normal.
\end{theorem}

\begin{proof}
Let $${\rm Cn} := \text{Spec} (S/I(\fh)^h) \subseteq (\mathbb C^2)^d$$ be the multi-affine cone of $\bar \fh.$  Let $$\text{pr}: {\rm Cn} \cap (\mathbb C^2 - \{0\})^d \rightarrow \bar \fh$$ be the natural projection map.  Zariski locally on $\bar \fh$, the variety ${\rm Cn} \cap (\mathbb C^2 - \{0\})^d$ is the fiber product of $\bar \fh$ with $(\mathbb C^{\times})^d.$  By Theorem \ref{CM}, ${\rm Cn}$ is Cohen-Macaulay.  Hence the open subset ${\rm Cn} \cap (\mathbb C^2 - \{0\})^d$ is also Cohen-Macaulay.  Then, by Theorem \ref{tensor}, $\bar \fh$ is Cohen-Macaulay.  It follows that $\bar \fh$ satisfies Serre's condition (S2).  Thus $\bar \fh$ is a normal variety by Theorem \ref{m3'} and Serre's criterion for normality.
\end{proof}

\begin{rmk}
A version of the above argument implies that matroid Schubert varieties are normal.
\end{rmk}

\subsection{Weyl Group Action on \texorpdfstring{$\bar \fh$}{hbar}}

Let $W$ be the Weyl group for $(\fg, \fh).$  The group $W$ acts on $(\mathbb P^1)^{\Phi}$ by 
\begin{align}\label{Waction}
w \cdot (x_{\lambda})_{\lambda \in \Phi} := (x_{w^{-1}(\lambda)})_{\lambda \in \Phi}.
\end{align}
With this $W$-action, the morphism $\fh \rightarrow (\mathbb P^1)^{\Phi}$ sending $h \in \fh$ to the point whose $\lambda$-component is $\lambda(h)$ is $W$-equivariant. 
From this it follows that $\bar \fh$ is a $W$-stable subset of $(\mathbb P^1)^{\Phi}.$  Hence we have proved

\begin{prop}
The usual $W$-action on $\fh$ can be extended to a $W$-action on $\bar \fh.$
\end{prop}

Let $(\mathbb P^1 - \{0\})^d$ be the affine chart of $(\mathbb P^1)^d$ consisting of points all of whose components are nonzero.  The intersection $\bar \fh \cap (\mathbb P^1 - \{0\})^d$ is exactly the reciprocal variety of $\fh$ studied in \cite{PS}, and is an affine open neighborhood of the most singular point $(\infty, \ldots, \infty)$ of $\bar \fh$, as is evident in the following example.

\begin{example} \label{ex:A2}
Let $\fg = \mathfrak {sl}_3.$  We have $$\Phi^+ = \{\alpha, \beta, \alpha + \beta\}.$$  So we have $${\rm emb}'': \fh \lhook\joinrel\longrightarrow (\mathbb P^1)^3, ~ h \longmapsto \bigl( \alpha(h), \beta(h), \alpha(h) + \beta(h) \bigr).$$  If $[x_0: x_1], [y_0: y_1], [z_0: z_1]$ are coordinates on $(\mathbb P^1)^3$, then it is not hard to see that $$I(\fh)^h = \langle x_1y_0z_0 + x_0y_1z_0 - x_0y_0z_1 \rangle.$$  Using this, and computing in affine coordinate charts of $(\mathbb P^1)^3$, one sees that $\bar \fh - \fh$ is the union $$(\mathbb P^1 \times \{\infty\} \times \{\infty\}) \cup (\{\infty\} \times \mathbb P^1 \times \{\infty\}) \cup (\{\infty\} \times \{\infty\} \times \mathbb P^1)$$ of three copies of $\mathbb P^1.$  Recall that the wonderful compactification of a Cartan subalgebra of $\mathfrak {sl}_2$ is isomorphic to $\mathbb P^1.$  The three copies of $\mathbb P^1$ above correspond exactly to the three good root subsystems $$\{\pm \alpha\}, \{\pm \beta\}, ~ \text{and} ~ \{\pm (\alpha + \beta)\}$$ of $\Phi$, each of which is of type $A_1.$  It follows from Theorem \ref{m1'} that the union of the codimension one strata is $$(\mathbb C \times \{\infty\} \times \{\infty\}) \cup (\{\infty\} \times \mathbb C \times \{\infty\}) \cup (\{\infty\} \times \{\infty\} \times \mathbb C).$$  The three copies of $\mathbb P^1$ intersect at the point $(\infty, \infty, \infty)$, which is the complement of the union $$\fh \cup (\bigcup \limits_{\substack {\Phi' ~ \text{is a good} \\ \text{root subsystem}}} \fh')$$ in $\bar \fh.$  In local coordinates $(X_0,Y_0,Z_0)$ around $(\infty, \infty, \infty)$, where $X_0 := x_0/x_1$ and similarly for $Y_0$ and $Z_0$, one checks easily that $\bar \fh$ is the vanishing locus of $Y_0Z_0 + X_0Z_0 - X_0Y_0$, so that
$\bar \fh$ is not regular and $(\infty, \infty, \infty)$ is the only singular point of $\bar \fh.$  It is also not hard to see that the boundary $\bar \fh - \fh$ of $\fh$ is \textbf{not} a divisor of normal crossings.

The Weyl group in this case is $S_3.$  The $S_3$-action on $\bar \fh$ is the usual action on $\fh$, permutes the three copies of $\mathbb P^1$ and fixes the point $(\infty, \infty, \infty).$

Let $\fg = \mathfrak{sl}_4.$  Similar computations as above using Macaulay2 show that $\bar \fh$ is \textbf{not} a local complete intersection.

For a general $\fg$, let $(\infty, \ldots, \infty)$ be the point of $(\mathbb P^1)^d$ all of whose components are $\infty.$  One can show that $(\infty, \ldots, \infty)$ is contained in every irreducible component of $\bar \fh - \fh.$
\end{example}


\section{Cohomology of \texorpdfstring{$\bar{\fh}$}{hbar} and the Coxeter Arrangement}

In this section, we establish a bijection between the poset given by the strata of $\bar \fh$ with the intersection lattice of the Coxeter arrangement defined by the root system $\Phi.$  As a consequence, we deduce the Betti numbers of $\bar \fh$ from known results in \cite{OT} in the classical cases, and from results in \cite{OS1} or alternatively SageMath computations in the exceptional cases.  
   
Recall the stratification $\mathcal C$ of $\bar \fh = \sqcup \ \mathring C(\Psi)$ from Corollary \ref{strat}, and note the following easy consequence of this stratification. 
For $\Psi$ a $k$-step good root subsystem of $\Phi$, we consider the corresponding homology fundamental class $\Fund_{\Psi} \in H_{2(r-k)}(\bar \fh, \mathbb Z)$ determined by the $(r-k)$-dimensional subvariety $C(\Psi)$, and the corresponding cohomology class $\xi_{\Psi} \in H^{2(r-k)}(\bar \fh, \mathbb Z)$ defined by the property that 
$\langle \xi_{\Psi}, \Fund_{\Theta} \rangle = \delta_{\Psi, \Theta}$ for each $k$-step good root subsystem $\Theta$ of $\Phi.$

\begin{corollary} \label{homologyhbar}
$$
H_{2(r-k)}(\bar \fh, \mathbb Z)=\bigoplus \Z\cdot \Fund_{\Psi} \quad \text{and} \quad H^{2(r-k)}(\bar \fh, \mathbb Z)=\bigoplus \Z \cdot \xi_{\Psi},$$
where each sum is over $k$-step good root subsystems $\Psi$ of $\Phi.$
\end{corollary}

Thus, in order to determine the Betti numbers of $\bar \fh$, we need to count the number of strata in $\mathcal C$ in each dimension.  After a previous version of this paper was posted on the arXiv, we learned from Matt Douglass that, based on results proved in that version of this paper, the counting can be deduced from work of Orlik and Terao \cite[$\S$6.4 and Appendix C]{OT}.  

\begin{definition}
    For $0 \le k \le r = \dim \fh$, we write $$f(\Phi, k)$$ for the number of codimension $k$ strata in $\mathcal C.$

    When the root system $\Phi$ is of type ${\rm Tp} (= A_r, B_r, C_r, D_r, E_6, E_7, E_8, F_4, G_2)$, we also write $f({\rm Tp},k)$ for $f(\Phi,k).$
\end{definition}

\subsection{\texorpdfstring{$\bar{\fh}$}{hbar} and the Coxeter Arrangement}



We start by giving another characterization of the $k$-step good root subsystems.

\begin{lemma} \label{k-step}
    Let $\Phi$ be a root system of rank $r$ and $k \in [0, r].$  Then the $k$-step good root subsystems of $\Phi$ are exactly the maximal (with respect to inclusion) closed root subsystems of $\Phi$ of rank $r - k.$
\end{lemma}

\begin{proof}
    We use induction on $k.$  The base case $k = 0$ is trivial and the case when $k=1$ is immediate from the definition.

    Assume that $k > 1.$  Suppose that we have a chain
    \begin{align*}
        \Phi = \Phi_0 \supseteq \Phi_1 \supseteq \cdots \supseteq \Phi_k
    \end{align*}
    of root systems such that $\Phi_i$ is a good root subsystem of $\Phi_{i - 1}$ for $i \in [0, k].$  By definition, $\Phi_k$ is of rank $r - k.$  Since a closed root subsystem of a closed root subsystem is closed, $\Phi_k$ is a closed root subsystem of $\Phi.$  Let $\Phi'$ be a closed root subsystem of $\Phi$ of rank $r - k$ such that $\Phi' \supseteq \Phi_k.$  We need to prove that $\Phi' = \Phi_k.$  Assume that there exists $\lambda \in \Phi' - \Phi_k.$  Since $\lambda \in \Phi = \Phi_0$, there exists $i \in [1, k]$ such that $\lambda \in \Phi_{i - 1}$ and $\lambda \notin \Phi_i.$  Let $\Phi'_i$ be the closed root subsystem of $\Phi_{i - 1}$ generated by $\lambda$ and $\Phi_i.$  The assumptions that $\Phi'$ and $\Phi_k$ have the same rank and that $\Phi' \supseteq \Phi_k$ imply that $$\Span_{\mathbb Q} \Phi' = \Span_{\mathbb Q} \Phi_k.$$  Since $\lambda \in \Span_{\mathbb Q} \Phi' = \Span_{\mathbb Q} \Phi_k$ and $\Span_{\mathbb Q} \Phi_k \subseteq \Span_{\mathbb Q} \Phi_i$, it follows that $$\Span_{\mathbb Q} \Phi'_i = \Span_{\mathbb Q} (\lambda, \Phi_i) = \Span_{\mathbb Q} \Phi_i.$$  Therefore, we have $$\rk \Phi'_i = \rk \Phi_i = \rk \Phi_{i - 1} - 1.$$  But then $\Phi'_i$ is a closed root subsystem of $\Phi_{i - 1}$ of rank $\rk \Phi_{i - 1} - 1$ which properly contains $\Phi_i$, contradicting the assumption that $\Phi_i$ is a good root subsystem of $\Phi_{i - 1}.$

    Conversely, suppose that $\Psi$ is maximal among all closed root subsystems of $\Phi$ of rank $r - k.$  Choose a maximal element $\Phi'$ of the set of all closed root subsystems of $\Phi$ of rank $r - k + 1$ that contain $\Psi.$  Since a closed root subsystem of a closed root subsystem is closed, any closed root subsystem of $\Phi'$ of rank $r - k$ is also a closed root subsystem of $\Phi$ of rank $r  - k.$  So if it contains $\Psi$, it must be equal to $\Psi$ by maximality of $\Psi.$  It follows that $\Psi$ is a good root subsystem of $\Phi'.$  By the induction hypothesis, $\Phi'$ is a $(k - 1)$-step good root subsystem of $\Phi.$  Hence $\Psi$ is a $k$-step good root subsystem of $\Phi.$
\end{proof}

An easy consequence of Corollary \ref{strat} and Lemma \ref{k-step} is

\begin{theorem} \label{strclo}
There is a bijection between the maximal elements of the set of closed root subsystems of $\Phi$ of rank $\rk \Phi - k$ and the set of codimension $k$ strata in $\mathcal C$ given by $\Psi \mapsto \mathring C(\Psi).$
\end{theorem}

Let $\mathcal A$ be the Coxeter arrangement.  In other words, $\mathcal A$ consists of the root hyperplanes in $\fh$, where by a root hyperplane, we mean a hyperplane in $\fh$ of the form
\[
\lambda^{\perp} := \{x \in \fh: \lambda (x) = 0\} \text{ for some }  \lambda \in \Phi^+.
\]

\begin{definition} \label{Def:interLat}
    The \textit{intersection lattice} of $\mathcal A$ is the geometric lattice $L(\mathcal A)$ whose underlying set is $$\{\bigcap \limits_{\lambda \in S} \lambda^{\perp}: S \subseteq \Phi^+\}.$$  For subspaces $X, Y \in L(\mathcal A)$, the partial order $\le$ on $L(\mathcal A)$ is defined by $$X \le Y ~ \text{if and only if} ~ X \supseteq Y.$$  The rank function $\rk$ on $L(\mathcal A)$ is defined by $$\rk(X) := \codim_{\fh} X.$$  The join operation on $L(\mathcal A)$ is defined by $$X \vee Y := X \cap Y.$$
\end{definition}

It is evident that the $W$-action on $\fh$ induces a $W$-action on $L(\mathcal A).$

\begin{definition}
    We define a ranked poset structure on the stratification $\mathcal C$ on $\bar \fh$ as follows.  For $C, C' \in \mathcal C$, we define $$C \le C' ~ \text{if and only if} ~ C ~ \text{is contained in the closure of} ~ C'.$$  The rank function is defined by $$\rk(C) := \dim C.$$
\end{definition}

Define maps
\begin{align*}
    F: L(\mathcal A) \longrightarrow \mathcal C, ~ X \longmapsto \mathring C(X^{\perp} \cap \Phi) \quad \text{and} \quad
    G: \mathcal C \longrightarrow L(\mathcal A), ~ \mathring C(\Psi) \longmapsto \Psi^{\perp}.
\end{align*}

\begin{prop}\label{cox}
The map $G: \mathcal C \to L(\mathcal A)$ is an isomorphism of ranked posets with inverse given by $F:L(\mathcal A) \to \mathcal C.$
\end{prop}




\begin{proof}
We first verify that the map $F$ is well-defined.  For $S \subseteq \Phi^+$, let $X = \bigcap \limits_{\lambda \in S} \lambda^{\perp}$ be the corresponding element of $L(\mathcal A).$  By Theorem \ref{strclo}, we need to verify that $X^{\perp} \cap \Phi$ is a maximal closed root subsystem of rank $\rk (X^{\perp} \cap \Phi).$  We have $$X^{\perp} \cap \Phi = (\bigcap \limits_{\lambda \in S} \lambda^{\perp})^{\perp} \cap \Phi = \Span_{\mathbb Q} (S) \cap \Phi.$$  It is clear that $\Span_{\mathbb Q} (S) \cap \Phi$ is a closed root subsystem of $\Phi$ of rank $\dim \Span_{\mathbb Q} S.$  Suppose $\Theta$ is a closed root subsystem of $\Phi$ of rank $\dim \Span_{\mathbb Q} S$ which contains $\Span_{\mathbb Q} (S) \cap \Phi.$  Then $\Span_{\mathbb Q} \Theta = \Span_{\mathbb Q} S$, so $$\Theta \subseteq \Span_{\mathbb Q} (\Theta) \cap \Phi = \Span_{\mathbb Q} (S) \cap \Phi.$$  This proves that $\Span_{\mathbb Q} (S) \cap \Phi$ is a maximal element of the set of closed root subsystems of $\Phi$ of rank $\dim \Span_{\mathbb Q} S$, so $F$ is well-defined.

Now we prove that $F(G(\mathring C(\Psi))) = \mathring C(\Psi).$  For this, we need to prove that $\Span_{\mathbb Q} (\Psi) \cap \Phi = \Psi.$ This follows from the obvious inclusion $\Psi \subseteq \Span_{\mathbb Q} (\Psi) \cap \Phi$ and maximality of $\Psi.$

Next we prove that $G(F(X)) = X.$   As above, letting $X= \bigcap \limits_{\lambda \in S} \lambda^{\perp}$, we have $$G(F(X)) = G \bigl( \mathring C(\Span_{\mathbb Q} (S) \cap \Phi) \bigr) = (\Span_{\mathbb Q} (S) \cap \Phi)^{\perp} \subseteq S^{\perp} = \bigcap \limits_{\lambda \in S} \lambda^{\perp} = X.$$ 
Also, we have
    \begin{align*}
        \dim G(F(X)) & = r - \dim \Span_{\mathbb Q} (\Span_{\mathbb Q} (S) \cap \Phi) \\ & = r - \dim \Span_{\mathbb Q} S = r - \dim X^{\perp} = \dim X.
    \end{align*}
    Combining these, we get $G(F(X)) = X.$

    To check that $F$ and $G$ respect the poset structures, we let $X_1 = \bigcap \limits_{\lambda \in S_1} \lambda^{\perp}, X_2 = \bigcap \limits_{\lambda \in S_2} \lambda^{\perp} \in L(\mathcal A)$ be such that $X_1 \le X_2.$  By definition, this means that $X_1 \supseteq X_2$, so that
 $X_1^{\perp} \subseteq X_2^{\perp}$, from which it follows that $$F(X_1) = \mathring C(X_1^{\perp} \cap \Phi) \subseteq C(X_2^{\perp} \cap \Phi) = F(X_2).$$  But this implies that $F(X_1) \le F(X_2)$ with respect to the partial order on $\mathcal C.$  Conversely, let $\mathring C(\Psi_1)$ and $\mathring C(\Psi_2) \in \mathcal C$ be such that $\mathring C(\Psi_1) \le \mathring C(\Psi_2).$  This just means that as subsets of $\Phi$, we have $\Psi_1 \subseteq \Psi_2.$  But then $\Psi_1^{\perp} \supseteq \Psi_2^{\perp}$, namely, $G(\mathring C(\Psi_1)) \le G(\mathring C(\Psi_2)).$

We compute
    \begin{align*}
        \rk (F(X)) = \rk (\Span_{\mathbb Q} (S) \cap \Phi) = \dim \Span_{\mathbb Q} S = \dim X^{\perp} = r - \dim X = \rk X.
    \end{align*}
    Hence $F$ respects the rank functions.
\end{proof}

\begin{theorem} \label{coxequivariant}
The map $G: {\mathcal C} \to L(\mathcal A)$ is a $W$-equivariant isomorphism of ranked posets with inverse $F: L(\mathcal A) \to {\mathcal C}.$
\end{theorem}

\begin{proof}  Proposition \ref{cox} implies all assertions except the $W$-equivariance.
For $x \in \bar \fh$ and $w \in W$, by Equation (\ref{Waction}), the $\lambda$-component of $w \cdot x$ is $x_{w^{-1} \cdot \lambda}.$  Hence the $\lambda$-component of $w \cdot x$ is not equal to infinity if and only if the $(w^{-1} \cdot \lambda)$-component of $x$ is not equal to infinity.   
It follows that $$w \cdot \mathring C(\Psi) = \mathring C(w\Psi).$$  Hence we have $$G(\mathring C(w\cdot \Psi)) = (w \cdot \Psi)^{\perp} = \bigcap \limits_{\lambda \in \Psi} (w \cdot \lambda)^{\perp} = \bigcap \limits_{\lambda \in \Psi} w \cdot (\lambda^{\perp}) = w \cdot \bigcap \limits_{\lambda \in \Psi} \lambda^{\perp} = w \cdot \Psi^{\perp} = w \cdot G(\mathring C(\Psi)),$$ which proves the theorem.
\end{proof}

Using Theorem \ref{coxequivariant}, we can interpret the Betti numbers of $\bar \fh$ in terms of $L(\mathcal A).$

\begin{definition}
    Let $L$ be a finite ranked poset with a unique minimal element $\hat 0.$  Write $\mu$ for the \textit{M\"obius function} of $L$, which is the unique integer-valued function $\mu$ on $L \times L$ such that $\mu(x,y)=0$ if $x\not\le y$, $\mu(x,x)=1$ for all $x\in L$, and $\sum_{\{ z: x\le z \le y \}} \mu(x,z)=0$ for all $x < y.$

    The \textit{characteristic polynomial} (resp. \textit{generating function}) $p_L$ (resp. $q_L$) of $L$ is
    \begin{align*}
        p_L(t) := \sum \limits_{X \in L} \mu(\hat 0, X) t^{\rk X} \quad
        (\text{resp.} ~ q_L(t) := \sum \limits_{X \in L} t^{\rk X}).
    \end{align*}

    For $0 \le k \le \rk L$, the \textit{Whitney number of the first (resp. second) kind} $w_k(L)$ (resp. $W_k(L)$) of $L$ is the coefficient of $t^{\rk L - k}$ in $p_L(t)$ (resp. $q_L(t)$).
\end{definition}

As a consequence of Theorem \ref{coxequivariant}, we have

\begin{corollary} \label{bettiwhitney}
    We have $$f(\Phi, k) = W_k(L(\mathcal A)).$$
\end{corollary}

\begin{rmk}
    The intersection lattice $L(\mathcal A)$ for types $B_r$ and $C_r$ are the same.  Thus, $f(B_r,k) = f(C_r,k).$ 
\end{rmk}

It is worth noticing that $w_k(L(\mathcal A))$ also has topological meaning.  To explain this, consider the complex manifold $$M := \fh - \bigcup \limits_{\lambda \in \Phi^+} \lambda^{\perp}.$$  The space $M$ was studied extensively by Artin \cite{Art}, Brieskorn \cite{Brie}, and Orlik and Solomon \cite{OS}.
For $0 \le k \le \dim \fh$, write $f'(\Phi, k)$ for the $(\dim \fh - k)$th Betti number of $M.$  Orlik and Solomon proved that

\begin{theorem} \cite[Theorem 5.2]{OS}
    We have $$f'(\Phi, k) = |w_k(L(\mathcal A))|.$$
\end{theorem}

The two kinds of Whitney numbers are related by a \textit{M\"obius inversion formula}, which is a form of duality in combinatorics.  It is interesting to look for a topological/geometric duality between the spaces $\bar \fh$ and $M$ that explains the combinatorial duality of their Betti numbers.

\subsection{Betti Numbers of \texorpdfstring{$\bar \fh$}{hbar}} \label{sect:BettiClassical}

In this section, we use the results of the previous section to determine the Betti numbers of $\bar{\fh}$ for each simple Lie algebra.    When the root system $\Phi$ is of classical type, we give the answer in terms of familiar combinatorial invariants using formulas from \cite{OT}, and when $\Phi$ is of exceptional type, we deduce the answers from \cite{OS1}.

\begin{definition}
    For $n \in \mathbb N$ and $1 \le k \le n$, the \textit{Stirling number} $S(n,k)$ \textit{of the second kind} is the number of partitions of the set $[1, n]:= \{1, \ldots, n\}$ into $k$ nonempty parts.  We set also $S(0,0)=1$ and $S(n,0)=0$ for $n > 0.$
\end{definition}

\begin{definition}
    Let $L$ be a finite group of order $m.$  For any $n \in \mathbb N$, we consider collections $\alpha_A=(A_1, A_2, \dots, A_k; a_1, a_2, \dots, a_k)$ where $A_1, \dots, A_k$ are mutually disjoint subsets of $[1,n]$, and  $a_i:A_i \to L$ is a map for each $i\in [1,k].$  If $\alpha_A$ is as above, and $\alpha_B=(B_1,\dots, B_k; b_1, \dots, b_k)$ is another collection, we say $\alpha_A$ is equivalent to $\alpha_B$ if for each $i\in [1,k]$,  there is $j\in [1,k]$ and  $g_i \in L$ such that $g_i\circ a_i=b_j.$  For $\alpha_A$ as above, we say the corank of $\alpha_A$ is equal to $k$, the number of mutually disjoint subsets in $A.$  The \textit{Dowling lattice} $Q_n(L)$ is the set of equivalence classes of these collections $\alpha_A.$

    For $0 \le k \le n$, the \textit{Whitney number} $W_k(Q_n(L))$ \textit{of the second kind} of $Q_n(L)$ is the number of corank $k$ elements of $Q_n(L).$
\end{definition}

\begin{example}
$\ $
    \begin{enumerate}
        \itemsep 0em
        \item If $L$ is the trivial group $\{e\}$, then $$W_k(Q_n(L)) = S(n+1,k+1).$$
        \item If $L$ is the group $\mathbb Z/2$, then $$W_k(Q_n(L)) = \text{the number of ``type} ~ B ~ \text{partial partitions'' of the set} ~ \{1, \ldots, n\} ~ \text{into} ~ k ~ \text{blocks}.$$
    \end{enumerate}
\end{example}

By Corollary \ref{bettiwhitney}, the Betti numbers of $\bar \fh$ when $\Phi$ is of classical type can be computed by counting the subspaces in the corresponding hyperplane arrangement, which are determined in \cite{OT}.  We are grateful to Matt Douglass for pointing out these formulas to us.  In a previous version of this paper, we developed an intricate combinatorics in order to count the $k$-step good root subsystems, which is now unnecessary in view of \cite{OT}.

We computed the Betti numbers for the exceptional Lie algebras using SageMath,  and are grateful to the referee for pointing out that they can be deduced from the tables in \cite{OS1}.  In fact, in \cite{OS1}, the row labeled $A_0$ counts parabolic root subsystems in the given root subsystem in a particular $W$-orbit of parabolic root subsystem.  Thus, the $2(r-k)$ Betti number of $\bar \fh$ can be computed from the corresponding table in \cite{OS1}, by summing over $W$-orbits with a parabolic root subsystem of rank $k$ in the table.  For example, the entry $24$ for $f(F_4,3)$ may be obtained from table V in \cite{OS1} by summing the entries in the columns for $A_1$ and $\tilde{A}_1$, and the entry $122$ for $f(F_4,2)$ by summing the entries from the 4 columns for rank 2 root systems in table V.

\begin{theorem} \label{thm:Betti}
$\ $
\begin{enumerate}
    \itemsep 0em
    \item \cite[Proposition 6.72 for type $A$, Proposition 6.76 for types $B$ and $C$, and Corollary 6.81 for type $D$]{OT}
    Let $\Phi$ be a root system of classical type of rank $r \in \mathbb N$ and $0 \leq k \leq r$.  The numbers $f(\Phi,k)$ are given by
    \begin{center}
        \begin{tabular}{|c||c|}
            \hline
            $\Phi$ & $f(\Phi,k)$ \\
            \hline \hline
            $A_r$ & $W_k(Q_r(\{e\})) = S(r+1,k+1)$ \\
            $B_r$ and $C_r$ & $W_k(Q_r(\mathbb Z/2)) = \sum \limits_{i=0}^{r-k} \binom{r}{i} S(r-i,k) 2^{r-k-i}$ \\
            $D_r$ & $\sum \limits_{i=k}^{r} \binom{r}{i} S(i,k) 2^{i-k} - r S(r-1, k) 2^{r-1-k}$ \\
            \hline
        \end{tabular}
    \end{center}
    \item \cite[tables V-VIII for $F_4,$ $E_6,$ $E_7$ and $E_8$]{OS1}  Let $\Phi$ be a root system of exceptional type.  The numbers $f(\Phi,k)$ are given by
    \begin{center}
        \begin{tabular}{|c||ccccccccc|}
            \hline
            \diagbox[height=1.5em]{$\Phi$}{$k$} & $0$ & $1$ & $2$ & $3$ & $4$ & $5$ & $6$ & $7$ & $8$ \\
            \hline \hline
            $E_6$ & $1$ & $639$ & $2001$ & $1530$ & $390$ & $36$ & $1$ && \\
            $E_7$ & $1$ & $8821$ & $36435$ & $33411$ & $10395$ & $1281$ & $63$ & $1$ & \\
            $E_8$ & $1$ & $440880$ & $2221780$ & $2091600$ & $661542$ & $85680$ & $4900$ & $120$ & $1$ \\
            $F_4$ & $1$ & $120$ & $122$ & $24$ & $1$ &&&& \\
            $G_2$ & $1$ & $6$ & $1$ &&&&&& \\
            \hline
        \end{tabular}
    \end{center}
\end{enumerate}
\end{theorem}

\vskip 5mm

\begin{example}
    We present the numbers $f(\Phi,k)$ when $\Phi$ is of classical type and $\rk \Phi$ is small in the tables below.

    \begin{table}[ht!]
    \begin{center}
        \caption{The numbers $f(A_r,k)$}
        \begin{tabular}{|c||cccccc|}
            \hline
            \diagbox[height=1.5em]{$r$}{$k$} & $0$ & $1$ & $2$ & $3$ & $4$ & $5$ \\
            \hline \hline
            $1$ & $1$ & $1$ &&&& \\
            $2$ & $1$ & $3$ & $1$ &&& \\
            $3$ & $1$ & $7$ & $6$ & $1$ && \\
            $4$ & $1$ & $15$ & $25$ & $10$ & $1$ & \\
            $5$ & $1$ & $31$ & $90$ & $65$ & $15$ & $1$ \\
            \hline
        \end{tabular}
    \end{center}
    \begin{center}
        \caption{The numbers $f(B_r,k) = f(C_r,k)$}
        \begin{tabular}{|c||cccccc|}
            \hline
            \diagbox[height=1.5em]{$r$}{$k$} & $0$ & $1$ & $2$ & $3$ & $4$ & $5$ \\
            \hline \hline
            $2$ & $1$ & $4$ & $1$ &&& \\
            $3$ & $1$ & $13$ & $9$ & $1$ &&\\
            $4$ & $1$ & $40$ & $58$ & $16$ & $1$ & \\
            $5$ & $1$ & $121$ & $330$ & $170$ & $25$ & $1$ \\
            \hline
        \end{tabular}
    \end{center}
    \begin{center}
        \caption{The numbers $f(D_r,k)$}
        \begin{tabular}{|c||ccccccc|}
            \hline
            \diagbox[height=1.5em]{$r$}{$k$} & $0$ & $1$ & $2$ & $3$ & $4$ & $5$ & $6$ \\
            \hline \hline
            $4$ & $1$ & $24$ & $34$ & $12$ & $1$ && \\
            $5$ & $1$ & $81$ & $190$ & $110$ & $20$ & $1$ & \\
            $6$ & $1$ & $268$ & $1051$ & $920$ & $275$ & $30$ & $1$ \\
            \hline
        \end{tabular}
    \end{center}
    \end{table}
\end{example}

The generating function of the numbers $f(\Phi,k)$ can be computed using Theorem \ref{thm:Betti}.

\begin{definition}
    Define the two variable generating function $F_A(q, t)$ to be the following formal sum: $$F_A(q, t) := \sum \limits_{r=0}^{\infty} \sum \limits_{k=0}^r f(A_r, k) q^k \frac{t^r}{(r+1)!}.$$
\end{definition}

Using the well-known formula $$\sum_{n=0}^{\infty} \sum_{k=0}^n S(n,k)\frac{x^n}{n!}y^k = e^{y(e^x - 1)}$$ for the Stirling numbers of the second kind, we obtain the following corollary of Theorem \ref{thm:Betti}.

\begin{corollary}
    We have $$F_A(q, t) = \frac{e^{q(e^t-1)} - 1}{qt}.$$
\end{corollary}

For types $B$, $C$ and $D$, the following slight modification of $F_A(q,t)$ is more convenient.

\begin{definition}
    Define the two variable generating functions $F_B(q, t)$, $F_C(q, t)$, and $F_D(q, t)$ to be the following formal sum:
    \[
    F_B(q, t) = F_C(q, t) := \sum \limits_{r=0}^{\infty} \sum \limits_{k=0}^r f(B_r, k) q^k \frac{t^r}{r!} \quad \text{and} \quad F_D(q, t) := \sum \limits_{r=0}^{\infty} \sum \limits_{k=0}^r f(D_r, k) q^k \frac{t^r}{r!}.
    \]
\end{definition}

The following corollary follows easily from Theorem \ref{thm:Betti} and \cite[Theorem 4]{Suterintseq}.

\begin{corollary}
    We have
    \[
    F_B(q, t) = F_C(q, t) = \exp(t + \frac{q(e^{2t} - 1)}{2}) \quad \text{and} \quad F_D(q, t) = (e^t - t) \exp(\frac{q(e^{2t} - 1)}{2}).
    \]
\end{corollary}

\begin{rmk}
    $\ $
    \begin{enumerate}
        \itemsep 0em
        \item If $\Phi$ is of type $A_r$, the Euler characteristic of $\bar \fh$, namely the number $\sum \limits_{k = 0}^r f(A_r, k),$ is referred to as the $(r + 1)$st \textit{Bell number} in the literature.
        \item If $\Phi$ is of type $B_r$ or $C_r$, the Euler characteristic of $\bar \fh$, namely the number $\sum \limits_{k = 0}^r f(B_r, k),$ is referred to as the $r$th \textit{Dowling number} in the literature.
        \item The sequences $\{f(\Phi, k): k \in [0, \rk \Phi]\}$ also appeared in the work of Suter \cite{Suterintseq} in the context of hyperplane arrangements.
    \end{enumerate}
\end{rmk}

\section{Topology of \texorpdfstring{$\bar{\fh}$}{hbar}}

In this section, we determine the Weyl group representation on $H^{\bullet}(\bar \fh, \C)$ and determine the cup product structure of $H^{\bullet}(\bar \fh, \Z).$
 
\subsection{Weyl Group Representation}

Recall the definition of $L(\mathcal A)$ in Definition \ref{Def:interLat}. 
Let $\mathbb C \cdot L(\mathcal A)$ be the $\mathbb C$-vector space with $L(\mathcal A)$ as a basis.  It is graded by declaring that $X \in L(\mathcal A)$ is in degree $2\rk X.$  It follows from Theorem \ref{coxequivariant} that

\begin{corollary} \label{representation}
    The maps $F$ and $G$ induce mutually inverse isomorphisms
    \begin{align*}
        F: \mathbb C \cdot L(\mathcal A) \longrightarrow H^{\bullet}(\bar \fh, \C) \\
        G: H^{\bullet}(\bar \fh, \C) \longrightarrow \mathbb C \cdot L(\mathcal A)
    \end{align*}
    of graded $W$-representations.

    In particular, $H^{\bullet}(\bar \fh, \C)$ is a permutation representation of $W.$
\end{corollary}

\begin{proof}
    We only need to observe that the $W$-action on $L(\mathcal A)$ is a permutation representation.
\end{proof}

For $X \in L(\mathcal A)$, define $$W^X := \{w \in W: w ~ \text{fixes every point of} ~ X\}.$$  By definition, $W^X$ is a parabolic subgroup of $W.$  By Steinberg \cite[Theorem 1.5]{Ste}, $W^X$ is the reflection group generated by the reflections associated to $\lambda \in X^{\perp}.$  The following theorem is due to Orlik and Solomon \cite{OS1}.

\begin{theorem} \cite[Lemmas 3.4 and 3.5]{OS1} \label{conj}
    For $X, Y \in L(\mathcal A)$, the following conditions are equivalent.
    \begin{enumerate}
        \itemsep 0em
        \item $X$ and $Y$ are in the same $W$-orbit in $L(\mathcal A)$;
        \item $W^X$ and $W^Y$ are conjugate in $W$;
        \item The Coxeter elements of $W^X$ and $W^Y$ are conjugate in $W.$
    \end{enumerate}
\end{theorem}


Moreover, let $N(W^X)$ be the normalizer of $W^X$ in $W.$  Orlik and Solomon \cite{OS1} also proved

\begin{prop} \cite[Lemma 3.4]{OS1} \label{stab}
    The stabilizer of the subspace $X$ in $W$ is $N(W^X).$
\end{prop}

Detailed analysis of the structure of the group $N(W^X)$ can be found in the work of Howlett \cite[Corollary 3]{How}.  In particular, Howlett proved that $N(W^X)$ is the semidirect product of $W^X$ with the subgroup of $W$ consisting of those elements which map a given set of Coxeter generators of $W^X$ to itself.

Following Orlik and Solomon \cite{OS1}, we refer to the $W$-conjugacy class of a Coxeter element of a parabolic subgroup of $W$ as a \textit{parabolic class}.  By Theorem \ref{conj}, the map
\begin{align*}
    \{W \text{-orbits in} ~ L(\mathcal A)\} & \longrightarrow \{\text{parabolic classes in} ~ W\} \\
    \quad W \cdot X & \longmapsto ~ \text{conjugacy class of a Coxeter element of} ~ W^X
\end{align*}
is a well-defined bijection.

The set $\mathcal P$ of parabolic classes in $W$ has been classified by many different authors \cite{C, DPR, OS1}.  For a parabolic class $c$, choose an element $X$ of $L(\mathcal A)$ such that $c$ is a Coxeter element of $W^X.$  This is always possible by the bijection in the previous paragraph.  For the sake of simplicity, we write $N(c)$ for the normalizer $N(W^X).$  We define the rank $\rk N(c)$ of $N(c)$ to be the rank of the parabolic subgroup $W^X.$  Although $N(c)$ is only well-defined up to conjugation in $W$, its rank is well-defined.


As a corollary of Theorem \ref{coxequivariant}, Theorem \ref{conj} and Proposition \ref{stab}, we have

\begin{corollary} \label{rep}
    There is an isomorphism of $W$-representations $$H^{\bullet}(\bar \fh, \C) \simeq \bigoplus \limits_{c \in \mathcal P} {\rm Ind}_{N(c)}^W \mathds 1,$$ where $\mathds 1$ stands for the trivial representation of $N(c).$

    Moreover, if one places ${\rm Ind}_{N(c)}^W \mathds 1$ in degree $2\rk N(c)$, then the above is an isomorphism of graded $W$-representations.
\end{corollary}


Let $\bar H$ be the toric variety associated with the fan whose maximal cones are the Weyl chambers.  It is known that there is a $W$-action on $\bar H$ extending the $W$-action on $H$, so there is an induced $W$-action on $H^{\bullet}(\bar H, \C).$  It is interesting to compare Corollary \ref{rep} with the following result of Stembridge \cite{St}.  In particular, it is interesting to relate the set $\mathcal P$ of parabolic classes to the $W$-orbits in the set $\mathcal Q$ that appears below, noting that $\mathcal P$ has cardinality at most $2^{\rk \Phi}$.

\begin{theorem} \cite[Theorem 4.1 and Corollary 4.6]{St}
    $H^{\bullet}(\bar H, \C)$ is a permutation representation of $W.$  For any set $\mathcal Q$ with a $W$-action which is a basis for $H^{\bullet}(\bar H, \C)$, there are $2^{\rk \Phi}$ $W$-orbits in $\mathcal Q.$
\end{theorem}

\subsection{Cup Product}

For $X \in L(\mathcal A)$, recall that $F(X)= \mathring C(\Psi)$ for some $(r-k)$-step good root subsystem $\Psi=\Psi_X$ of $\Phi$ and let 
$k = \dim(\mathring C(\Psi)).$
 Let  $\Fund_X =[C(\Psi_X)] \in H_{2k}(\bar \fh, \Z)$ denote the cellular homology class represented by the stratum $F(X) \in \mathcal C,$ and let $\{ \xi_X : X \in L(\mathcal A)\}$ denote the dual basis of cohomology.
  Recall that
\[
H_{\bullet}(\bar \fh, \Z) = \bigoplus \limits_{X \in L(\mathcal A)} \mathbb Z \cdot \Fund_X \quad \text{and} \quad H^{\bullet}(\bar \fh, \Z) = \bigoplus \limits_{X \in L(\mathcal A)} \mathbb Z \cdot \xi_X.
\]
The cup product of the $\xi_X$'s is computed in the following

\begin{theorem}[see also {\cite[Theorem 14]{HW}}] \label{cupprod}
    For $X, Y \in L(\mathcal A)$, we have
    \begin{align*}
        \xi_X \smile \xi_Y =
        \begin{cases}
            \xi_{X \vee Y}  & ~ \text{if} ~ \rk(X \vee Y) = \rk X + \rk Y \\
            0 & ~ \text{else}.
        \end{cases}
    \end{align*}
In particular, $H^{\bullet}(\bar \fh, \Z)$ is generated by its degree $2$ component.
\end{theorem}

To prove this result, we recall some standard facts about intersection theory on the smooth variety $(\mathbb P^1)^d,$ where $d=|\Phi^+|.$  If $\eta \subseteq \Phi^+$, we consider the subvariety
$$
Q_{\eta} =\{(x_{\lambda} ) \in (\mathbb P^1)^d : x_{\lambda}=0 \ \forall \lambda \not\in \eta \}.
$$
We note that $Q_{\eta} \cong (\mathbb P^1)^{|\eta|}$ and let 
$[Q_{\eta}] \in H_{2|\eta|} ((\mathbb P^1)^d, \Z)$ be the homology class defined by the submanifold $Q_{\eta}$,  and let $\{ u_{\eta} \}$ denote the corresponding dual basis of
 $H^{\bullet}((\mathbb P^1)^d, \Z).$   Then $H^2((\mathbb P^1)^d, \Z)$ is generated by the elements $u_{\lambda} \in H^2((\mathbb P^1)^d, \Z),$ and 
\[
H^{\bullet} ((\mathbb P^1)^{d}, \Z) \cong \prod \limits_{\lambda \in \Phi^+} \mathbb Z[u_{\lambda}]/(u_{\lambda}^2) \cong \bigoplus \limits_{\substack{\{ \lambda_1, \ldots, \lambda_k\} \subseteq \Phi^+ }} \mathbb Z \cdot u_{\lambda_1} \cdots u_{\lambda_k}.
\]
There is a Poincar\'e duality isomorphism $\mathcal D:H_{2\bullet}((\mathbb P^1)^d, \Z) \to H^{2(d-\bullet)}((\mathbb P^1)^d, \Z)$, and 
$\mathcal D([Q_{\eta}])=u_{\Phi^+ - \eta}.$   We consider also the intersection pairing 
\[
H_{2k}((\mathbb P^1)^d, \Z) \otimes H_{2\ell}((\mathbb P^1)^d, \Z) \longrightarrow H_{2(k + \ell - d)}((\mathbb P^1)^d, \Z), \ [X]\cdot [Y]:=(\mathcal D)^{-1}(\mathcal D(X)\smile \mathcal D(Y)), 
\]
for $[X] \in H_{2k}((\mathbb P^1)^d, \Z)$ and $[Y] \in H_{2\ell}((\mathbb P^1)^d, \Z).$
Further, for $\eta_1, \eta_2 \subseteq \Phi^+,$
   \begin{align*}
       u_{\eta_1} \smile u_{\eta_2}= 
       \begin{cases} 
         u_{\eta_1 \cup \eta_2} & ~ \text{if} ~ \eta_1 \cap \eta_2 = \emptyset \\
         0                 & ~ \text{else},
     \end{cases}
     \end{align*}
and if $|\eta_2|=d-|\eta_1|,$
    \begin{equation}\label{pairing}
    \begin{split}
       [Q_{\eta_1}]\cdot [Q_{\eta_2}]=
       \begin{cases}
       [pt] & ~ \text{if} ~ \eta_2 = \Phi^+ - \eta_1   \\
       0                            & ~ \text{else}.
    \end{cases}
    \end{split}
    \end{equation}

As a consequence, if $\eta = \{ \lambda_1, \dots, \lambda_k \}$, then
\begin{equation*}
u_{\eta}= \prod \limits_{i=1, \dots, k} u_{\lambda_i}.
\end{equation*}

Note also that there is a cycle map $A^{\bullet}((\mathbb P^1)^d) \to H^{2\bullet}((\mathbb P^1)^d, \Z)$ which is a ring isomorphism (\cite[Example 19.1.11(b) and Corollary 19.2(b)]{Ful}).  In particular, we can use assertions about the product in the Chow ring to compute the intersection pairing in homology.
  
Consider the natural inclusion $\iota: \bar \fh \hookrightarrow (\mathbb P^1)^d.$   For a $(r-k)$-step good root subsystem $\Psi$ of $\Phi^+$, let $LI(\Psi)$ be the set of $k$-element subsets $\eta = \{ \lambda_1, \dots, \lambda_k \}$ such that $\lambda_1, \dots, \lambda_k$ is a basis of a Cartan subalgebra $\fh_{\Psi}$ contained in $\fh.$  Recall that $\fh_{\Psi}$ is a Cartan subalgebra of a semisimple Lie algebra $\fg_{\Psi}$ such that $\Psi$ is the roots of $\fg_{\Psi}$ with respect to $\fh_{\Psi}.$
We prove

\begin{lemma} \label{lift}
Consider $X \in L(\mathcal A)$ and $\mathring C(\Psi)=F(X).$  Then the class 
\[
\iota_{\ast}([C(\Psi)])=\sum_{\eta \in LI(\Psi)} [Q_{\eta}] \in H_{2k}((\mathbb P^1)^d, \Z).
\]
Further, if $\eta = \{ \lambda_1, \dots, \lambda_k \} \in LI(\Psi)$, we have 
\[
\iota^{\ast}(\prod \limits_{i=1, \dots, k} u_{\lambda_i})=\iota^{\ast}(u_{\eta})=\xi_X.
\]
and if $\eta = \{ \lambda_1, \dots, \lambda_k \} \subseteq \Psi$ is linearly dependent, then
\[
\iota^{\ast}(\prod \limits_{i=1, \dots, k} u_{\lambda_i}) = \iota^{\ast}(u_{\eta}) = 0.
\]
\end{lemma}

\begin{proof}
For the first assertion, by Equation (\ref{pairing}), it suffices to prove that the intersection pairing $\iota_{\ast}([C(\Psi)])\cdot [Q_{\theta}] = 1$ when $\Phi^+ - \theta \in LI(\Psi)$ and is $0$ for each other subset $\theta$ of $\Phi^+$ of cardinality $d-k.$  Suppose first that $\theta = \Phi^+ - \eta$ where $\eta \in LI(\Psi)$ and we compute the pairing $\iota_{\ast}(C(\Psi)) \cdot Q_{\theta}.$   By assumption, any point in the intersection has $x_{\lambda_i}=0$ for each $\lambda_i \in \eta.$  Then for each root $\lambda$ of $\Psi$, there exist rational numbers $b_1, \dots, b_k$ such that  $\lambda = \sum_{i=1}^k b_i \lambda_i.$  Thus, $x_{\lambda} - \sum b_i x_{\lambda_i}$ is a linear polynomial that vanishes on $\iota(\fh_{\Psi})$, and it follows that $x_{\lambda}=0.$   Hence, the only intersection point of $\iota(C(\Psi))$ and $Q_{\theta}$ is $\iota(0)$ in the smooth open subvariety $\iota(\fh_{\Psi})$ of $\iota(C(\Psi))$, and further $\iota(0)$ is affine locally the intersection of linear subspaces. It follows that the coefficient of $[Q_{\eta}]$ in $\iota_*([C(\Psi)])$ equals $1$ by Theorem 1.26 of \cite{EH}.

If $\theta$ is a subset of cardinality $|\Phi^+|-k$ which is not of the form
$\Phi^+ - \eta$ for any $\eta \in LI(\Psi)$, then either $\Phi^+ - \theta$ contains a positive root $\beta$ not in $\Psi$, or $\Phi^+ - \theta = \{ \lambda_1, \dots, \lambda_k \}$ is contained in $\Psi$ but is linearly dependent.  In the
first case,  the value of the $\beta$-coefficient is $0$ for any point of $Q_{\theta}$, but the value of the $\beta$-coefficient is $\infty$ for  any point in  $\iota(C(\Psi)).$  Hence, the intersection is empty so the intersection pairing is $0$.   In the second case, consider a nontrivial linear relation $c_1 \lambda_1 + \dots + c_k \lambda_k = 0,$ and choose $x_{0,\lambda_1}, \dots, x_{0,\lambda_k} \in \C$ 
such that $\sum c_i x_{0, \lambda_i}\not=0.$  Then we can move $Q_{\theta}$ without changing its homology class so that the $\lambda_i$ coordinate of any point of $Q_{\theta}$ is $x_{0,\lambda_i}$, and such a point is not in $C(\Psi).$  Hence, in the second case, the intersection pairing is also $0$.   The first assertion follows.

The last two assertions follows from the first assertion using functoriality of the pairing between cohomology and homology and the fact that $H_{\bullet}(\bar \fh, \Z)$ is spanned freely by the classes $[C(\Psi)].$
\end{proof}

\begin{proof}[Proof of Theorem \ref{cupprod}]
    Choose a basis $\lambda_1, \ldots, \lambda_p \in \Phi^+$ (resp. $\mu_1, \ldots, \mu_q \in \Phi^+$) of $X^{\perp}$ (resp. $Y^{\perp}$).  By Lemma \ref{lift},
$\iota^{\ast}(u_{\lambda_1} \cdots u_{\lambda_p})=\xi_X$ and $\iota^{\ast}(u_{\mu_1} \cdots u_{\mu_q})=\xi_Y.$

    Suppose $\rk (X \vee Y) = \rk X + \rk Y$, namely, $\codim (X \cap Y) = \codim X + \codim Y.$  Then $$\dim (X^{\perp} + Y^{\perp}) = \dim (X \cap Y)^{\perp} = \dim X^{\perp} + \dim Y^{\perp},$$  so that $\lambda_1, \ldots, \lambda_p, \mu_1, \ldots, \mu_q$ is a basis of $(X \vee Y)^{\perp}.$  It follows that 
$$\xi_{X \vee Y} = \iota^{\ast}(u_{\lambda_1} \cdots u_{\lambda_p} u_{\mu_1} \cdots u_{\mu_q}) = \iota^{\ast}(u_{\lambda_1} \cdots u_{\lambda_p}) \smile \iota^{\ast} (u_{\mu_1} \cdots u_{\mu_q}) = \xi_X \smile \xi_Y,$$ which proves the assertion in this case.

    Conversely, suppose $\rk (X \vee Y) \neq \rk X + \rk Y$, so that $\dim (X^{\perp} + Y^{\perp}) < \dim X^{\perp} + \dim Y^{\perp}.$  Then the elements $\lambda_1, \ldots, \lambda_p, \mu_1, \ldots, \mu_q$ are linearly dependent.  Hence by Lemma \ref{lift}, we deduce that $\iota^{\ast}(u_{\lambda_1} \ldots u_{\lambda_p} u_{\mu_1} \cdots u_{\mu_q})= 0.$  It follows that $\xi_X \smile \xi_Y = 0.$
\end{proof}

\noindent Sam Evens: Department of Mathematics, University of Notre Dame, 255 Hurley, Notre Dame, IN 46556.  Email: sevens@nd.edu

\medskip 

\noindent Yu Li: Department of Mathematics, University of Notre Dame, 255 Hurley, Notre Dame, IN 46556.  Email: yli234@nd.edu

\end{document}